\documentclass[english,11pt, a4paper]{article}

\usepackage[T1]{fontenc}

\usepackage{amssymb}
\usepackage{amsmath}
\usepackage{amsfonts}
\usepackage{amsthm}
\usepackage{mathtools}
\usepackage{mathabx}  

\usepackage[dvipsnames]{xcolor}
\usepackage{graphicx}
\usepackage{float}

\usepackage{mdframed}
\usepackage{enumitem} 

\usepackage[export]{adjustbox}
\usepackage{bbold}
\usepackage[percent]{overpic}
\usepackage{units}

\usepackage{booktabs,array}
\usepackage{url} 
\usepackage{hyperref}

\hypersetup{
    colorlinks=true,
    citecolor=purple!70!blue,
    linkcolor=black,
    filecolor=black,      
    urlcolor=black,
    pdftitle={},
    }

\usepackage{pgfplots}
\pgfplotsset{compat=1.6}

\usetikzlibrary{calc,quotes,angles,arrows.meta,positioning, shapes}
\usetikzlibrary{decorations.pathreplacing,decorations.markings}

\usepackage{float}

\usepackage{caption}
\usepackage{subcaption}

\usepackage[nocompress]{cite}

\theoremstyle{plain}%
\newtheorem{theorem}{Theorem}[section]
\newtheorem{lemma}[theorem]{Lemma}
\newtheorem{proposition}[theorem]{Proposition}

\newtheorem*{conjecture*}{Conjecture}

 \numberwithin{equation}{section}

\theoremstyle{definition}
\newtheorem{definition}[theorem]{Definition}

\theoremstyle{remark}
\newtheorem{remark}[theorem]{Remark}

 \let \leq \leqslant
 \let \geq \geqslant

\DeclareMathOperator{\tr}{tr}
\DeclareMathOperator{\real}{Re}

\DeclareMathOperator{\support}{supp}

\DeclareMathOperator{\dist}{dist} 
 
\DeclareMathOperator*{\argmin}{arg\, min}

\usepackage{fancyhdr}
\usepackage{graphics}

\definecolor{detailcolor00}{rgb}{0.4405, 0.204, 0.343}
\definecolor{detailcolor01}{rgb}{0.546, 0.215, 0.352}
\definecolor{detailcolor02}{rgb}{0.675, 0.247, 0.387} 
\definecolor{detailcolor03}{rgb}{0.775, 0.317, 0.455}
\definecolor{detailcolor04}{rgb}{0.830, 0.421, 0.553} 
\definecolor{detailcolor05}{rgb}{0.831, 0.533, 0.663}
\definecolor{detailcolor06}{rgb}{0.779, 0.619, 0.775}
\definecolor{detailcolor07}{rgb}{0.724, 0.694, 0.827}
\definecolor{detailcolor08}{rgb}{0.687, 0.770, 0.880}
\definecolor{detailcolor09}{rgb}{0.671, 0.839, 0.904}
\definecolor{detailcolor10}{rgb}{0.659, 0.872, 0.882}

\usepackage{scalerel,stackengine}
\newcommand\pig[1]{\scalerel*[5.5pt]{\Big#1}{%
  \ensurestackMath{\addstackgap[1.5pt]{\big#1}}}}
\newcommand\pigl[1]{\mathopen{\pig{#1}}}
\newcommand\pigr[1]{\mathclose{\pig{#1}}}

\title{Pure perimeter laws for Wilson lines observables}

\author{Malin P. Forsstr\"om\thanks{Mathematical Sciences, Chalmers University of Technology and University of Gothenburg, SE-412 96 Gothenburg, Sweden } \mbox{}\footnote{palo@chalmers.se}}

\begin{document}

\maketitle 

\begin{abstract}
	Several recent papers have studied the decay rate of the expectation of Wilson loop and Wilson line observables in lattice gauge theory and the lattice Higgs model. These results have all been perturbative in the sense that the parameters need to scale with the length of the loop or line for the error term to be smaller than the estimate. In this paper, we further develop ideas from~\cite{fv2023} and~\cite{f2024} to give more detailed asymptotics for the expectation of Wilson loop and Wilson line observables, which do not require the parameters to be very large or very small for the error term to be small, thus improving the results of~\cite{flv2022, flv2023, flv2020, f2022b}. In particular, we show that Wilson line and loop observables have a pure perimeter law in the Higgs and confinement phases of the lattice Higgs model.
\end{abstract}

%\textbf{Keywords:}	70S15, 81T13, 81T25, 82B20

\section{Introduction}
 
Lattice gauge theories are a family of spin models defined on the edges of hypercubic lattices with interaction between edges in the boundary of some common plaquette. 
The lattice Higgs model is a lattice gauge theory coupled to an external field that is a simple model of a Higgs field. In the simplest case, the model is described by the Hamiltonian
\begin{equation}\label{eq: Hamiltonian}
	\beta \sum_{p \in C_2(B_N)}   \rho_G(d\sigma(p)) + \kappa \sum_{e \in C_1(B_N)}   \tr \rho_G(\sigma(e)) \rho_H(d\phi(e))^{-1}
\end{equation}
together with a uniform reference measure. 
Here \( \beta,\kappa \geq 0 \) are two parameters,  \( B_N \subseteq \mathbb{Z}^d \) is a box of side length \( N, \) \( C_1(B_N) \) is the set of all oriented edges in \( B_N \), \( C_2(B_N) \) is the set of oriented plaquettes in \( B_N \), \( G \) and \( H \) are Lie groups, \( \sigma \in \Omega_1(B_N,G) \) is a \( G \)-valued 1-form,  \( \phi \in \Omega_0(B_N,H) \) is a \( H \)-valued 0-form,  \ and \( \rho_G \) and \( \rho_H \) are one-dimensional representations of \( G \) and \( H \) respectively. We let \( \mu_{\beta,\kappa,N} \) be the corresponding measure, and let \( \mathbb{E}_{N,\beta,\kappa} \) be the corresponding expectation.
Elements in \( \sigma \in \Omega_1(B_N,G) \) will be referred to as \emph{gauge field configurations}, and elements in \( \phi \in \Omega_0(B_N,H) \) will be referred to as \emph{Higgs field configurations}. We think of the pair \( (\sigma , \phi )\) as describing spin configurations where a group element \( \sigma(e) \in G \) is attached to each edge \( e\) of the lattice and a group element \( \phi(v) \in H \) is attached to each vertex of the lattice. The measure described by~\eqref{eq: Hamiltonian}  thus gives a probability measure on such spin configurations. 
The simplest example of a lattice Higgs model is obtained by letting \( G=H=\mathbb{Z}_2, \) and this model will be referred to as the Ising lattice Higgs model.

We note that in the physics literature, one always has \( G=H , \) while the more general setting here is considered in, e.g., \cite{a2021}. Moreover, the case \( G \neq H \) can also be considered a toy model for the abelian lattice Higgs model with charge, see, e.g., \cite{ghms2011,fs1979}.

Two natural observables for lattice gauge theories and lattice Higgs models are Wilson loops and Wilson lines. Given a path \( \gamma \) of oriented edges, the Wilson line observable associated to \( \gamma \) is defined by
\begin{equation}
	W_\gamma \coloneqq  \rho_G(\sigma(\gamma)) \rho_H (\phi(\partial \gamma))=  \prod_{e \in \gamma} \rho_G(\sigma(e)) \prod_{v \in \partial \gamma} \rho_H(\phi(v)).
\end{equation}
If \( \gamma \) is a closed loop,  the Wilson line observable is also referred to as a Wilson loop observable, and in this case, we have
 \begin{equation}
	W_\gamma \coloneqq  \rho_G(\sigma(\gamma)) =  \prod_{e \in \gamma} \rho_G(\sigma(e)) .
\end{equation}

By the Ginibre inequality, for any Wilson line or Wilson loop observable \( W_\gamma, \) the limit \( \langle W_\gamma \rangle_{\beta,\kappa} \) of \( \mathbb{E}_{N,\beta,\kappa}[W_\gamma] \) as \( N \to \infty \) exists and is translation invariant (see, e.g.,~\cite{flv2020}.

In this paper, we will make the additional assumption that  \( G \) and \( H \) are cyclic groups, i.e., that there is \( m,n \geq 0 \) finite such that \( G = \mathbb{Z}_m \) and \( H = \mathbb{Z}_n, \) and in this case we can assume that the representations are the one-dimensional representations \( j \mapsto e^{ij\pi}. \)
In this case, since~\eqref{eq: Hamiltonian} describes a spin model with block spins that take values in a finite group, the GHS inequalities hold, and hence by standard arguments (see, e.g., \cite{fv2017}) a limiting probability measure on \( \mathbb{Z}^d \) exists. This will let us transfer results from a finite lattice to an infinite lattice.

Wilson loop and Wilson lines observables are known to undergo phase transitions and can thus be used as order observables. In particular, when \( \kappa = 0, \) the expected value of Wilson loop observables are known to obey a perimeter law for large~\( \beta\) and an area law for small~\( \beta. \) In other words, there are constants \( C_\beta \) and \( C_\beta' \) such that for sufficiently large~\( \beta, \) \( \langle W_\gamma \rangle_{\beta,0} \sim e^{-C_\beta |\gamma|},\) and for sufficiently small \( \beta, \) \( \langle W_\gamma \rangle_{\beta,0}  \sim e^{-C_\beta' \mathrm{area}(\gamma)}, \) where \( \mathrm{area}(\gamma) \) is the cardinality of the smallest set of plaquettes that has the loop \( \gamma \) as its boundary. In contrast, when \( \kappa >0, \) Wilson loop observables are known always to obey a perimeter law (see, e.g., Theorem~\ref{theorem: general perimeter} below). When \( \kappa=0, \) the expectation of Wilson line observables corresponding to open paths is always zero. However, when \( \kappa>0, \) the expected value of Wilson line observables is non-trivial, and Wilson line observables are known to undergo phase transitions when \( \kappa>0 \) (see, e.g.,~\cite{bf1983} and also~\cite{f2024}), which we now describe.
To this end, let \( \gamma_n \) be a straight path of length \( n \) and let the parameters \( \beta,\kappa \geq 0 \) be fixed.
We say that the model has \emph{perimeter law decay} if the limit 
\[
 	a_{\beta,\kappa} \coloneqq \lim_{n \to \infty }\frac{-\log \langle W_{\gamma_n}\rangle_{\beta,\kappa}}{|\gamma_n|}
 \]
exists and is strictly positive. 
The model is said to have \emph{pure perimeter law decay} if the limit 
\begin{equation}\label{eq: Cbetakappa}
 	C_{\beta,\kappa} \coloneqq \lim_{n \to \infty } -\log \langle W_{\gamma_n}\rangle_{\beta,\kappa}-a_{\beta,\kappa}|\gamma_n| 
\end{equation} 
exists and is strictly positive. In contrast, the model is said to have \emph{perimeter law decay with lower-order corrections} if the limit in~\eqref{eq: Cbetakappa} is zero and there is a function \( g \colon \mathbb{N}\to \mathbb{R} \) such that \( \lim_{n \to \infty} |g(n)| = \infty, \) \( \lim_{n\to \infty} g(n)/n = 0 , \) and the limit
\[
\lim_{n \to \infty} -\log \langle W_{\gamma_n} \rangle_{\beta,\kappa} -a_{\beta,\kappa} |\gamma_n| - g(|\gamma_n|) 
\]
exists.
%\[
%\langle W_{\gamma_n} \rangle_{\beta,\kappa} \sim e^{-a_{\beta,\kappa} |\gamma_n|-g(|\gamma_n|)}.
%\]
The model is said to have \emph{perimeter law decay with polynomial corrections} if there is a polynomial \( p(x) \) such that \( g(|\gamma|_n) =  \log p(|\gamma_n|). \) In this case, we have 
\[
\langle W_{\gamma_n} \rangle_{\beta,\kappa} \sim \frac{e^{-a_{\beta,\kappa} |\gamma_n|}}{p(|\gamma_n|)}.
\]
This type of decay is known as \emph{Ornstein-Zernike decay} and holds, e.g., in the high-temperature regime of the Ising model (see, e.g.,~\cite{pl1978,civ2003}).
 
The lattice Higgs model, with \( G \) and \( H \) finite, is believed to have pure perimeter law decay in some parts of the phase diagram, referred to as the \emph{Higgs/confinement regime}, and perimeter law decay with polynomial corrections in other parts of the phase diagram, referred to as the \emph{free phase} (see Figure~\ref{fig: phase diagram})~\cite{c1980,fs1979,ghms2011}. (See also~\cite{abp2025,ghms2011} for a discussion about potential differences between the Higgs and confinement regime.) The main contribution of this paper is to verify that the model indeed has pure perimeter law decay in the Higgs and confinement regimes, thus verifying some parts of this picture.

\begin{figure}[tb]\centering
		\begin{tikzpicture}[scale=1.4] 
  		\begin{axis}[xmin=-0.05,xmax=0.8,ymin=-0.05,ymax=0.8,xticklabel=\empty,yticklabel=\empty,axis lines = middle,xlabel={\footnotesize $\beta$},ylabel={\footnotesize $\kappa$},label style =
               {at={(ticklabel cs:1.08)}}, color=detailcolor07, ticks=none]

               \addplot[domain=0:0.32, samples=100, color=detailcolor00, dashed] {0.61-2*x^1.8};
               
               \addplot[domain=0.32:0.42, samples=100, color=detailcolor00] {0.61-2*x^1.8};
               
               \addplot[domain=0.42:0.44, samples=100, color=detailcolor00]{0.61-2*0.42^1.8-100*(x-0.42)^1.6};
               \addplot[domain=0.42:0.9, samples=100, color=detailcolor00]{0.61-2*0.42^1.8-0.1*(x-0.42)+0.14*(x-0.42)^2};
  		\end{axis}
  		
  		\draw (1.9,1.9) node[align=center] {\footnotesize Confinement phase \\ \scriptsize pure perimeter law decay};
  		\draw (5.3,0.9) node[align=center] {\footnotesize Free phase \\ \scriptsize perimeter law decay with \\[-1.1ex] \scriptsize polynomial corrections};
  		\draw (4.2,3.7) node[align=center] {\footnotesize Higgs phase \\ \scriptsize pure perimeter law decay};

	\end{tikzpicture}

	\caption{The phase diagram of the Ising lattice Higgs model as described in the physics literature~\cite{c1980,fs1979,ghms2011}. In the Higgs and confinement phases, the expectation of Wilson line observables is believed to have pure perimeter law decay. In contrast, in the free phase, one expects the expectation of Wilson line observables to have perimeter law decay with polynomial corrections.}\label{fig: phase diagram}
\end{figure}
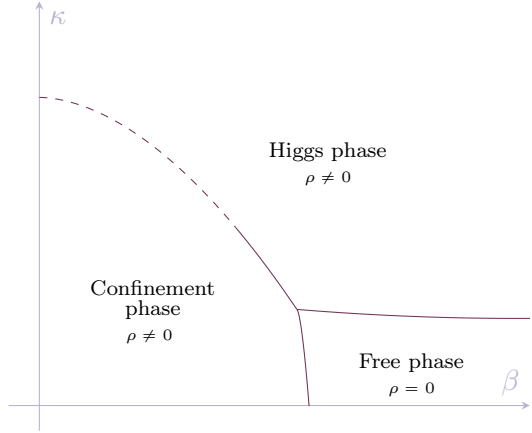
 
\subsection{Results}

In this section, we present our main results. Our first result shows that Wilson line and loop observables obey a perimeter law for all cyclic groups \( G \) and \( H \) when \( \kappa >0. \) In particular, it never has a so-called area law decay, which is known to hold for Wilson loop observables when \( \kappa = 0 \) and \( \beta \) is sufficiently small.

\begin{theorem}\label{theorem: general perimeter}
	For \( m,n \geq 2, \) let \( G = \mathbb{Z}_m \) and \( H = \mathbb{Z}_n. \) Further, let \( \beta > 0 \) and \( \kappa > 0. \) Then the following holds. 
	\begin{enumerate}[label=(\roman*)]
		\item There is a constant \(  b_{\beta,\kappa }  \in (0,\infty) \) (which depend on \( m \) and \( n \)) such that for any path \( \gamma, \)
	\begin{equation*}
		e^{-b_{\beta,\kappa}|\gamma|} \leq \langle W_\gamma\rangle_{\beta,\kappa} \leq 1.
	\end{equation*}\label{item: general perimeter a}
	\item Let \( a,b \in \mathbb{N}. \) If, for each \( j \geq 1, \)  \( \gamma_j \) is a straight path of length \( j \) or a rectangular loop with side lengths \( aj \) and \( bj ,\) then \( \lim_{j \to \infty} \frac{-\log \langle W_{\gamma_j} \rangle_{\beta,\kappa}}{|\gamma_j|} \in (0,\infty)\) exists.\label{item: general perimeter b}
	\end{enumerate} 
\end{theorem} 

The assumption on the shape of \( \gamma_j \) in Theorem~\ref{theorem: general perimeter}\ref{item: general perimeter b} is not strictly necessary. Rather, it is sufficient to have a sequence of paths that is increasing in length while having the same shape in some sense, and the proof can easily be adapted to this setting.

The proof of Theorem~\ref{theorem: general perimeter} can without much work be extended to finite abelian groups, and we outline such a proof in Appendix~\ref{appendix: general theory}.
However, since the proof of this result use Griffiths second inequality, a proof for general finite groups would require a different approach.

We now present our main results for the Higgs and confinement phases. These results state that in both regimes, Wilson line and Wilson loop observables corresponding to straight lines and rectangular loops have a pure perimeter law.
In both theorems below, the constants \( a_{\beta,\kappa} ,C_{\beta,\kappa} ,D_{\beta,\kappa}, D_{\beta,\kappa}'\) that appear are given by concrete expansions in the corresponding proofs.

\begin{theorem}\label{theorem: Higgs phase}
	Let \( G=H=\mathbb{Z}_2, \) and for each \( n \geq 1, \) let \( \gamma_n \) be a straight path of length~\( n \) or a rectangular loop with side lengths \( an \) and \( bn \) for some \( a,b \in \mathbb{N}. \)
	Then there is \( \kappa_0^{(Higgs)}>0 \), independent of~\( a,b \) and~\( n \), such that if \( \beta > 0 \) and  \( \kappa > \kappa_0^{(\textrm{Higgs})} \) 
	\begin{equation}\label{eq: inequality Higgs}
		\bigl| -\log \langle W_{\gamma_n} \rangle_{\beta,\kappa}  - a_{\beta,\kappa} |\gamma_n| - C_{\beta,\kappa} \bigr| 
		\leq D_{\beta,\kappa}  e^{-D_{\beta,\kappa}' |\gamma_n| } .
	\end{equation}  
	for some constants \( C_{\beta,\kappa}, a_{\beta,\kappa}, D_{\beta,\kappa},D_{\beta,\kappa}' >0.\) 
\end{theorem}

\begin{theorem}\label{theorem: main confinement}
	Let \( G=H=\mathbb{Z}_2, \) and for each \( n \geq 1, \) let \( \gamma_n \) be a straight path of length \( n \) or a rectangular loop with side lengths \( an \) and \( bn \) for some \( a,b \in \mathbb{N}. \)
	Then there is~\(\beta_0^{(\text{conf})}>0\), independent of~\( a,b \) and~\( n \), such that if \( 0 < \beta < \beta_0^{(\text{conf})} \) and  \(  \kappa >0, \) then
 	\begin{equation}\label{eq: inequality confinement}
 		\bigl| -\log \langle W_{\gamma_n} \rangle - a_{\beta,\kappa} |\gamma_n| - C_{\beta,\kappa} \bigr| \leq D_{\beta,\kappa}e^{-D_{\beta,\kappa}' |\gamma_n|} .
	\end{equation}
	for some constants \( C_{\beta,\kappa}, a_{\beta,\kappa}, D_{\beta,\kappa},D_{\beta,\kappa}' >0.\)
\end{theorem}

\begin{remark}
	In the settings of Theorem~\ref{theorem: Higgs phase} and Theorem~\ref{theorem: main confinement}, it follows from~\eqref{eq: inequality Higgs} and~\eqref{eq: inequality confinement} respectively that
	\[
 		\bigl| \langle W_{\gamma_n} \rangle - e^{-a_{\beta,\kappa} |\gamma_n| - C_{\beta,\kappa}} \bigr| \leq Ce^{-C' |\gamma_n|+ Ce^{-C' |\gamma_n|}}e^{-a_{\beta,\kappa} |\gamma_n| + C_{\beta,\kappa}}.
	\]
	This upper bound holds throughout the Higgs and confinement regimes and does not require the constants \( \beta \) and \( \kappa \) to tend to infinity or zero as \( |\gamma| \to \infty. \) Theorem~\ref{theorem: Higgs phase} thus greatly improves~\cite[Theorem 1.1]{f2022b}, while Theorem~\ref{theorem: main confinement} is a substantial strengthening of~\cite[Theorem 1.3]{flv2022}. Moreover, the proofs can easily be modified to give the analog results for connected paths along the boundary of rectangles, and in this case, substantially strengthen the main results of and the case \( G=H=\mathbb{Z}_2 \) in~\cite{flv2023}. Moreover, with slightly more work, the results of this paper can also be extended to finite abelian groups with \( G=H\) (see Appendix~\ref{appendix: general theory}, thus improving the main results of~\cite{flv2020}. In detail, for Theorem~\ref{theorem: Higgs phase}, the proof extends immediately to this case, and for Theorem~\ref{theorem: main confinement}, as for Theorem~\ref{theorem: general perimeter}, the proof extends immediately if one first obtains a high temperature expansion for this general setting. 
	Finally, we mention that using the cluster expansion of~\cite{fv2023}, the methods developed in this paper can also be used to get the analog results in the pure model with \( \kappa = 0, \) which would then substantially strengthen the main results of~\cite{c2019,sc2019, flv2020} in the case \( G=\mathbb{Z}_2, \) by removing the assumption that \(\beta \to \infty \) with the length of the loop.
\end{remark}

\begin{remark}
	To obtain a scaling limit of lattice gauge theories where Wilson loop and line observables have non-trivial limits, one presumably needs to let either \( \beta \) or \( \kappa \) grow with the size of the loop or line considered. All of our results are also valid in this setting. However, in the physics literature, lattice gauge theories are well-studied models also away from this limit~\cite{d2025, fs1979, a1983, c1980, ghms2011, ssn2021}, and studying them in greater generality increases the understanding of other related models of statistical physics. From this point of view, taking \( \beta \) and \( |\gamma| \) large is insufficient to fully understand the model.
\end{remark}

\begin{remark}
	The thresholds \( \beta_0^{(conf)} \) and \( \kappa_0^{(Higgs)}\) for \( \beta \) and \( \kappa \) that appear in Theorem~\ref{theorem: Higgs phase} and Theorem~\ref{theorem: main confinement}, respectively, correspond to the largest and smallest values of the corresponding parameters respectively for which the cluster expansion is absolutely convergent, which in turn corresponds to the size of the polymers in the corresponding expansion having exponential decay (see also~\cite[Chapter 5]{fv2017}). This latter property is suggested as the driving force behind the model's phase transitions, see, e.g., \cite{ghms2011}. However, in both regimes of the model considered, this property should depend on both parameters. Since the thresholds used in the theorems here depend on only one parameter, they should not be sharp in general.
\end{remark}

\subsection{Outline of proof}\label{sec: proof idea}
The main ideas of the proofs of Theorem~\ref{theorem: Higgs phase} and Theorem~\ref{theorem: main confinement} are as follows. The first step is to express the logarithm of the expectation of a Wilson line observable in terms of a sum over clusters using a cluster expansion as in~\cite{f2024} (see also~\cite{fv2023}). In the confinement phase, to do this, we must first use a high-temperature expansion to translate the relevant observables to a related low-temperature model. In both cluster expansions, the logarithm of the Wilson line observable is affected only by clusters whose edge support intersects the corresponding path. 
The main contribution to the Wilson line (or loop) expectation comes from clusters that are small compared to the length of the path. From the perspective of a small cluster, most edges in the path look the same, and hence each type of cluster gives an equal contribution for each edge, except for edges that are close to the boundary (or corners) of the path. This gives a contribution of order \( |\gamma_n|.\) 
 After this contribution has been accounted for, the next order of contributions is caused by clusters that are sufficiently close to exactly one of the endpoints (or corners) of the path for their contribution to be affected by the endpoint. This gives a contribution of order one. Finally, the next-order contributions are given by clusters whose edge support is very close to both endpoints (or several of the corners) of the path. To achieve this, the cardinality of the support of the cluster has to be at least \( |\gamma_n|, \) and this gives a term that is exponentially small.

The papers~\cite{f2024,fv2023} used the first step of the outline above to conclude perimeter law, but did not include subsequent steps to be able to deduce the finer asymptotic behavior for obtained in this paper.

We mention that in the free phase, when \( \kappa  \) is small and \( \beta \) is large, the above proof idea fails. The reason for this is that the model obtained after performing a high-temperature expansion (in \( \kappa\) only) in this phase results in a model that is forced to contain a random path that connects the two endpoints of \( \gamma_n.\) Even a first-order estimate of \( \langle W_{\gamma_n}\rangle_{\beta,\kappa}\) requires a fine analysis of the geometry of this path.

\subsection{Relation to other work}

Several recent papers~\cite{flv2022, flv2023, flv2020, f2022b, c2019, sc2019, a2021} consider lattice gauge theories or lattice Higgs models with finite gauge groups. In all of these papers, the aim is to describe the leading order asymptotics of \( \langle W_{\gamma_n} \rangle_{\beta,\kappa} \) in a dilute gas limit, with \( \gamma_n \) being a loop or path along the boundary of a rectangle. A dilute gas limit means that the parameters \( \beta \) or \( \kappa \) are chosen such that they go to either zero or infinity as the length of the path goes to infinity in such a way that the density of plaquettes assigned a non-zero spin goes to zero with the length of the path. Such results are also said to be perturbative. In contrast, the main results of this paper do not require the parameters to be very small or very large compared to the length of the path in order to get a small error term. As a consequence, our results substantially strengthen the results of~\cite{flv2022, flv2023, flv2020, f2022b, c2019, sc2019, a2021} in the special case that \( G \) and \( H \) are both cyclic groups. In addition, the error term obtained by this technique is smaller than the error term obtained in these papers also in the dilute gas limit. While the methods of~\cite{flv2022, flv2023, flv2020, f2022b, c2019, sc2019, a2021} are relatively similar, the current paper uses very different techniques that seem more suitable in a non-dilute setting, and that in instead similar to the techniques employed in~\cite{fv2023}~and~\cite{f2024}.

In~\cite{fv2023}, non-perturbative results for the decay of Wilson loops in the low-temperature regime were found by using cluster expansions.
In~\cite{f2024}, these ideas were adapted to the abelian lattice Higgs model to show that the Marcu-Fredenhagen ratio, which is a limit of a ratio of Wilson line expectations whose length goes to infinity, undergoes a phase transition and can thus be used as an order parameter to distinguish between the phases of the lattice Higgs model. The ideas used in this paper are a further development of the ideas and techniques of~\cite{fv2023,f2024}.
We stress that the main contribution of these papers is not the cluster expansions, as these are special cases of the cluster expansion as presented in, e.g., \cite{fv2017}, and cluster expansions have been used to study the decay rate of Wilson loop observables also in the physics literature (see, e.g., \cite{s1980,d1973,b1984,bf1985}). 
In particular, in~\cite[Theorem 3.18]{s1980}, a sketch of how to use cluster expansions to obtain perimeter law is presented. However, we stress that our main result is not that Wilson loops obey a perimeter law but rather that in the Higgs/confinement regimes, there are no polynomial corrections to the perimeter law. In other words, our main result describes the behavior of the term remaining after the perimeter law term has been accounted for.

In~\cite{f2024}, we used cluster expansions to show that the Marcu-Fredenhagen ratio was non-zero in the Higgs/confinement regime. Moreover, in~\cite {fv2023}, we used cluster expansions to find an expansion for the counterpart of the constant \( a_{\beta,\kappa} \) in~\eqref{eq: inequality Higgs} and~\eqref{eq: inequality confinement} for the case \( \kappa = 0. \) In both these papers, the main contribution was new symmetry arguments applied to the terms in the cluster expansions which were needed to obtain the results (see also Section~\ref{sec: proof idea} below). 
%In detail, we showed that the main contribution to the cluster expansion came from edge or plaquette clusters of finite size, and for any such fixed cluster \( \sigma_0\) that is not to close to the corners of \( \gamma_n\), for large enough rectangular loops, there are of order \( \gamma_n \) other clusters that are translations or rotations of \( \sigma_0 \) that interact with \( \gamma_n \) in the same way as \( \sigma_0. \) This gives perimeter law with an explicit sum for the constant \( C_{\beta,\kappa}. \)

However, the methods used there, applied to the setting of this paper, were too coarse to obtain the asymptotics of Theorem~\ref{theorem: Higgs phase} and~Theorem~\ref{theorem: main confinement}. The main contribution of this paper is a substantial refinement of these symmetry arguments, which, in addition to~\eqref{eq: inequality Higgs} and~\eqref{eq: inequality confinement}, also explain what the different order asymptotics are caused by. These methods could easily be adapted to yield even finer asymptotics.

\color{black}

\subsection{Structure of the paper}

In Section~\ref{section: preliminaries}, we collect the notation and results that we will need for the proofs of our main results. In particular, in Section~\ref{section: high temperature in both parameters}, we review the model arising from a high-temperature expansion of the lattice Higgs model, and in Section~\ref{section: cluster expansions}, we recall the cluster expansions from~\cite{f2024} which we will use.
In Section~\ref{section: general perimeter law}, we use the high-temperature expansion from Section~\ref{section: high temperature in both parameters} to give a proof of Theorem~\ref{theorem: general perimeter}.
Finally, in Section~\ref{section: proofs of the main results}, we give proofs of Theorem~\ref{theorem: Higgs phase} and Theorem~\ref{theorem: main confinement}.

\section{Preliminaries}\label{section: preliminaries}

\subsection{Discrete exterior calculus}

The following notation from discrete exterior calculus will be used throughout the paper. For a more thorough introduction to this notation, we refer the reader to~\cite{flv2023}.

For \( N \geq 1, \) we let \( B_N \subseteq \mathbb{Z}^d \) be a box of side length \( N. \)

\begin{itemize}
	\item For \( k \in \{ 0,1, \dots, d \}, \) we let \( C_k (B_N) \) denote the set of all oriented \( k \)-cells of \( B_N, \) and let \( C_k (B_N)^+ \) denote the set of all positively oriented cells in \( C_k (B_N) .\) In particular, \( C_1(B_N) \) is the set of oriented edges in \( B_N \) and \( C_2(B_N) \) is the set of oriented plaquettes in \( B_N .\)
	
	\item For \( k \in \{ 0,1, \dots, d \}, \) we let \( C^k(B_N,\mathbb{Z}) \) denote the set of all integer-valued \( k \)-chains on \( C_k(B_N). \) In  particular, when \( k = 1\), then this is the set of all formal sums of positively oriented edges with coefficients in \( \mathbb{Z}. \)
	
	\item For \( k \in \{ 0,1, \dots, d \}, \) we let \( \Omega_k(B_N,G) \) denote the set of all \( k \)-forms on \( C_k(B_N). \) When \( c \in C^k(B_N,\mathbb{Z}), \) and \( \omega \in \Omega_k(B_N,G), \) we write \( \omega(c) = \sum_{p \in c} c(p)\omega(p). \) 
	\item When \( k \in \{ 1,2,\dots,d \}, \) we let \( \partial c \) be the \( (k-1) \)-chain corresponding to the oriented boundary of \( c. \) For \( k \in \{ 0,1, \dots, d \} \) and \( c \in C_k(B_N), \) we let \( \hat \partial c \coloneqq \sum_{c' \colon c \in \partial c'} c' .\)
	
	\item We let \( d \) denote the discrete exterior derivative mapping \( \Omega_k(B_N,G) \) to \( \Omega_{k+1}(B_N,G), \) and let \( \delta \) be the co-derivative mapping \( \Omega_k(B_N,G) \) to \( \Omega_{k-1}(B_N,G).\)
	 
	\item When \( k \in \{ 0,1,\dots, n \} \) and \( \omega \in \Omega_k(B_N,G), \) we let \( \support \omega = \{ c \in C_k(B_N) \colon \omega(c) \neq 0  \}. \)
\end{itemize}

\subsection{Standing notation}

To simplify notation, for a path \( \gamma \in C^1(B_N,\mathbb{Z}) , \) we let
\begin{equation*}
	Z[\gamma] \coloneqq \sum_{\substack{\sigma\in \Omega_1(B_N,\mathbb{Z}_m)\\ \phi\in \Omega_0(B_N,\mathbb{Z}_n)}}  \rho_G(\sigma(\gamma)) \rho_H(\phi(\partial \gamma))^{-1}e^{\beta \sum_{p \in C_2(B_N)} \rho_G(d\sigma(p)) } e^{\kappa\sum_{e \in C_1(B_N)} \rho_G(\sigma(e))\rho_H(d\phi(e))^{-1}}.
\end{equation*}

Whenever the group is clear from the context, we will suppress the subscript on \( \rho\) to simplify notation.

\subsection{High temperature expansions}\label{section: high temperature in both parameters}

In this section, we present a high-temperature expansion which will be useful in the confinement phase. The high temperature expansion is a well-known tool in the literature, but we include it here for \( G= \mathbb{Z}_m \) and \( H = \mathbb{Z}_n \) for completeness. The proof is similar to the proof for the case \( G=H=\mathbb{Z}_n \) in~\cite[Proposition 4.1]{flv2022}, and in this case, we recover~\cite[Proposition 4.1]{f2022b}. However, even in this case, this proof is slightly different since it does not involve first passing to a unitary gauge (see Section~\ref{sec: unitary gauge}).

\subsubsection{The high temperature expansion when \( G=\mathbb{Z}_m \) and \( H = \mathbb{Z}_n \)}\label{sec: high temperature ZmZn}

The main purpose of this section is to present a high-temperature expansion of the lattice Higgs model for the case when the gauge group and the Higgs group are general cyclic groups (not necessarily the same).

For \( m,n\geq 2, \) we let \( m \lor n \) denote the smallest number that is divisible by both \( m \) and \( n. \)
For \( n \geq 2, \) \( i \in \mathbb{Z}_n \) and \( a \geq 0, \) we let
\[
	\varphi_{ a,n}(i)  \coloneqq \sum_{j=0}^\infty \frac{a^{nj+i}}{(nj+i)!}
\]
where we abuse notation by identifying \( \mathbb{Z}_n \) and \( \{ 0,1,\dots, n-1 \} \) in the natural way. 
Further, for \( j \in \mathbb{Z}_{n}, \) we let 
\begin{align*}
		 \bar \varphi_{a,n}(j)  \coloneqq \sum_{i,i' \in \mathbb{Z}_{n}\colon i-i'=j} \varphi_{a,n}(i)\varphi_{a,n}(i') = \sum_{i \in \mathbb{Z}_{n}} \varphi_{a,n}(i)\varphi_{a,n}(i-j).
\end{align*}

Note that for any \( j \neq 0, \) we have \( \lim_{a \to 0}  \bar \varphi_{a,n}(j) =0 \) and \( \bar\varphi_{a,n}(j) = \bar \varphi_{a,n}(n-j). \)

The next result, Proposition~\ref{proposition: high high mn} below, describes the high-temperature expansion that we will use throughout the paper to study the confinement regime.
\begin{proposition}\label{proposition: high high mn}
	Let \( m,n \geq 2, \) and let \( G = \mathbb{Z}_m \) and \( H = \mathbb{Z}_n.\) Further, let \( \beta,\kappa \geq 0,\)  let \( \gamma \in C^1(B_N,\mathbb{Z}) ,\) and let \( \eta_\gamma \in \Omega_1(B_N,\mathbb{Z}_{m\vee n}) \) be defined by \( \eta_\gamma(e) = \gamma[e] \mod m\vee n. \)  Then 
	\begin{equation*}\begin{split}
		Z[\gamma]  & \propto
		\sum_{\substack{\omega' \in  \Omega_2(B_N,\mathbb{Z}_{m\vee n}) \\ \omega \in  \Omega_2(B_N, \mathbb{Z}_m)}}
		\prod_{p \in C_2(B_N)^+} \bar \varphi_{\beta,m}(\omega(p)) \prod_{e \in C_1(B_N)^+} \bar \varphi_{\kappa,m\lor n}((\delta \omega'-\gamma)(e))  
		\\&\qquad\qquad\cdot \mathbf{1}(\delta(\omega' + \omega) =0 \pmod m)  ,
	\end{split}\end{equation*} 
	where the proportionality constant depends on \( m, \) \( n, \) and \( N ,\) but not on \( \gamma. \)
\end{proposition}

\begin{proof} 
	For \( m \geq 1 \) and \( k \in \{ 0,1 \dots, d\}, \) let \( \Upsilon_{k,m}\) denote the set of all \( {\{ 0,1,\dots, m-1 \} }\)-valued functions on \( C_k(B_N) \), and let \( \Upsilon_{k,m}^+\) denote the set of all \( {\{ 0,1,\dots, m-1 \} }\)-valued \( k \)-chains on \( C_k(B_N)^+. \) 
	
	Let \( \sigma \in \Omega_1(B_N,\mathbb{Z}_m) \) and \( \phi \in \Omega_0(B_N,\mathbb{Z}_n). \)
	Note that for any \( n \geq 2, \) \( a \geq 0 , \) and \( z  \) a \( n\)th root of unity, we have
	\begin{equation}\label{eq: first Fourier expansion}
		e^{a z} 
		= \sum_{j=0}^\infty \frac{(a z)^j}{j!} 
		= \sum_{i=0}^{n-1} z^{ i} \sum_{j=0}^\infty \frac{a  ^{nj+i}}{(nj+i)!}  
		= \sum_{i=0}^{n-1} \varphi_{a,n}(i) z^{ i}  .% = \sum_{i \in \mathbb{Z}_n} \varphi_{a,n}(i) z^{ i} .
	\end{equation}
	Hence 
	\begin{align*}
		&e^{ \beta\sum_{p \in C_2(B_N)} \rho(d\sigma(p))} 
		=
		\prod_{p \in C_2(B_N)} e^{\beta \rho(d\sigma(p))}
		= 
		\prod_{p \in C_2(B_N)} \sum_{i =0}^{m-1} \varphi_{\beta,m}(i) \rho (d\sigma(p))^i
		\\&\qquad=
		\sum_{i \in \Upsilon_{2,m}} \prod_{p \in C_2(B_N)} \varphi_{\beta,m} (i(p) ) \rho (d\sigma(p))^{i(p)}
		\\&\qquad=
		\sum_{i \in \Upsilon_{2,m}} \prod_{p \in C_2(B_N)^+} \varphi_{\beta,m} (i(p) )\varphi_{\beta,m} (i(-p) ) \rho (d\sigma(p))^{i(p)-i(-p)} 
		\\&\qquad=
		\sum_{j \in \Upsilon_{2,m}^+} \prod_{p \in C_2(B_N)^+} \bar \varphi_{\beta,m} (j(p) ) \rho (d\sigma(p))^{j(p)}.
		%\\&\qquad=
		%\sum_{\omega \in \Omega_2(B_N,\mathbb{Z}_m)} \prod_{p \in C_2(B_N)^+} \bar \varphi_{\beta,m} (\omega(p) ) \rho (d\sigma(p))^{\omega(p)}.
	\end{align*}  
	Similarly, since  for any \( g\in \mathbb{Z}_m \) and \( h \in \mathbb{Z}_n, \) \( \rho(g)\rho(h)^{-1} \) is a \( (m \vee n) \)th root of unity, we have
	\begin{align*}
		&e^{\kappa \sum_{e \in C_1(B_N)} \rho(\sigma(e))\rho(d\phi(e))^{-1}} 
		= 
		\prod_{e \in C_1(B_N)} e^{\kappa \rho(\sigma(e))\rho(d\phi(e))^{-1}} 
		\\&\qquad= 
		\prod_{e \in C_1(B_N)} \sum_{i =0}^{m \vee n-1} \varphi_{\kappa,m\lor n}(i ) (\rho(\sigma(e))\rho(d\phi(e))^{-1})^{ i}  
		\\&\qquad= 
		\sum_{i \in \Upsilon_{1,m \vee n}} \prod_{e \in C_1(B_N)} \varphi_{\kappa,m\lor n}(i(e)) \prod_{e \in C_1(B_N)}  (\rho(\sigma(e))\rho(d\phi(e))^{-1})^{ i(e)} 
		\\&\qquad= 
		\sum_{i \in \Upsilon_{1,m \vee n}} \prod_{e \in C_1(B_N)^+} \varphi_{\kappa,m\lor n}(i(e)) \varphi_{\kappa,m\lor n}(i(-e)) 
		 \rho(\sigma(e))^{ i(e)-i(-e)} 
		\rho(d\phi(e))^{ -(i(e)-i(-e))} 
		\\&\qquad= 
		\sum_{k \in \Upsilon_{1,m \vee n}^+} \prod_{e \in C_1(B_N)^+} \bar \varphi_{\kappa,m\lor n}(k(e)) 
		 \rho(\sigma(e))^{ k(e)} 
		\rho(d\phi(e))^{ -k(e)} .
	\end{align*} 
	Combining the above observations, we get
	\begin{align*}
		&
		\rho(\sigma(\gamma)) \rho(\phi(\partial \gamma))^{-1}e^{\beta \sum \rho(d\sigma(p)) }e^{\kappa\sum_e \rho(\sigma(e))\rho(d\phi(e))^{-1}}
		\\&= 
		\sum_{\substack{j \in \Upsilon_{2,m}^+ \\k \in \Upsilon_{1,m\vee n}^+}} 
		\rho(\sigma(\gamma)) \rho(\phi(\partial \gamma))^{-1}
		\\&\qquad\qquad\cdot \prod_{p \in C_2(B_N)^+} \bar \varphi_{\beta,m} (j(p) ) \rho (d\sigma(p))^{j(p)}
		\prod_{e \in C_1(B_N)^+} \bar \varphi_{\kappa,m\lor n}(k(e)) 
		 \rho(\sigma(e))^{ k(e)} 
		\rho(d\phi(e))^{ -k(e)} 
		\\&= 
		\sum_{\substack{j \in \Upsilon_{2,m}^+ \\k \in \Upsilon_{1,m \vee n}^+}} 
		\prod_{p \in C_2(B_N)^+} \bar \varphi_{\beta,m} (j(p) ) \prod_{e \in C_1(B_N)^+} \bar \varphi_{\kappa,m\lor n}(k(e)) 
		\\& \qquad\qquad\cdot 
		 \prod_{v \in C^0(B_N)^+} \rho(\phi(v))^{-\partial \gamma[v] -  \partial k( v)}  
		\prod_{e \in C_1(B_N)^+}  
		 \rho(\sigma(e))^{ \gamma[e] + k(e) +   \partial  j(e)}  .
	\end{align*}
	Now fix \( j \in \Upsilon_{2,m}^+\) and \( k \in \Upsilon_{1,m\vee n}^+   \) and consider the sum
	\begin{align*}
		S_{j,k} \coloneqq \sum_{\substack{\sigma\in \Omega_1(B_N,\mathbb{Z}_m)\\ \phi\in \Omega_0(B_N,\mathbb{Z}_n)}} \prod_{v \in C^0(B_N)^+} \rho(\phi(v))^{-\partial \gamma[v] -  \partial k( v)}  
		\prod_{e \in C_1(B_N)^+}  
		 \rho(\sigma(e))^{ \gamma[e] + k(e) +   \partial j( e)}  .
	\end{align*}
	By symmetry, \( S_{j,k} \) is identically zero unless \( \partial \gamma + \partial k  \equiv 0 \pmod n \) and \( {\gamma+k + \partial j  \equiv 0 \pmod m}, \) and in this case, we have 
	\[ 
		S_{j,k} = |\Omega_1(B_N,\mathbb{Z}_m)| \cdot |\Omega_0(B_N,\mathbb{Z}_m)| = m^{|C_1(B_N)^+|} n^{|C_0(B_N)^+|} . 
		\]
	Combining the previous equations, we obtain
	\begin{align*}
		Z[\gamma]  &= \sum_{\substack{\sigma\in \Omega_1(B_N,\mathbb{Z}_m)\\ \phi\in \Omega_0(B_N,\mathbb{Z}_n)}}  \rho(\sigma(\gamma)) \rho(\phi(\partial \gamma))^{-1}e^{\beta \sum \rho(d\sigma(p)) }e^{\kappa\sum_e \rho(\sigma(e))\rho(d\phi(e))^{-1}} 
		\\&= 
		m^{|C_1(B_N)|} n^{|C_0(B_N)^+|}
		\\&\qquad\qquad\cdot
		\sum_{\substack{j \in \Upsilon_{2,m}^+\\k \in \Upsilon_{1,m\vee n}^+ }} 
		\prod_{p \in C_2(B_N)^+} \bar \varphi_{\beta,m} (j(p) ) \prod_{e \in C_1(B_N)^+} \bar \varphi_{\kappa,m\lor n}(k(e)) 
		\\&\qquad\qquad\cdot
		\mathbf{1}(\partial \gamma +  \partial k   = 0 \pmod n )
		\mathbf{1}(\gamma + k + \partial  j  =0 \pmod m)
		.
	\end{align*} 
	We will now use three observations. First, recall that \( \eta_\gamma \in \Omega_1(B_N,\mathbb{Z}_{m\vee n}) \) was defined by \( \eta_\gamma(e) = \gamma[e] \mod m\vee n ,\) and thus \( \partial \gamma = 0 \mod n \) is equivalent to that \( \delta \eta_\gamma = 0 \mod n. \) Next, we note that there is a natural bijection \( \Upsilon_{2,m}^+ \to \Omega_2(B_N,\mathbb{Z}_m) \). Finally, we note that the mapping \( {k \mapsto (k - \gamma \mod m \vee n)}\) is a bijection from \( \Upsilon_{1,m\vee n}^+ \to \Omega_1(B_N,\mathbb{Z}_{m\vee n}) \). Using these observations, it follows that 
	\begin{align*}
		Z[\gamma]  &= 
		m^{|C_1(B_N)|} n^{|C_0(B_N)^+|}
		\\&\qquad\qquad\cdot
		\sum_{\substack{\omega \in \Omega_2(B_N,\mathbb{Z}_m) \\ \eta  \in \Omega_1(B_N,\mathbb{Z}_{m\vee n}) }} 
		\prod_{p \in C_2(B_N)^+} \bar \varphi_{\beta,m} (\omega(p) ) \prod_{e \in C_1(B_N)^+} \bar \varphi_{\kappa,m\lor n}((\eta-\eta_\gamma)(e)) 
		\\&\qquad\qquad\cdot
		\mathbf{1}(\delta \eta= 0 \pmod n )
		\mathbf{1}(\eta +\delta \omega =0 \pmod m)
		.
	\end{align*}
	Comparing the above equation with the previous equation, we have let \( j \mapsto \omega , \) \( k \mapsto \eta-\eta_\gamma, \) \( \partial \gamma + \partial k \mapsto \delta \eta_\gamma + \delta (\eta - \eta_\gamma)  = \delta \eta, \) and \( \gamma+k+\partial j \pmod m \mapsto \eta_\gamma + \eta-\eta_\gamma + \partial \omega \pmod m= \eta + \partial \omega \pmod m .\)
	Now note that if \( \eta \in  \Omega_1(B_N,\mathbb{Z}_{m\vee n}) \) is such that \( \eta +\delta \omega =0 \pmod m \) and \( \delta \eta  = 0 \pmod n , \) then \( \delta \eta = 0 \pmod {m\lor n}. \) By the Poincar\'e lemma (see, e.g.,~\cite[Lemma 2.2]{c2019}), when \( \eta \in \Omega_1(B_N,\mathbb{Z}_{m\lor n}) \) is such that \( \delta \eta = 0, \) then there is \( \omega'\in \Omega_2(B_N,\mathbb{Z}_{m\lor n})  \) such that \( \delta \omega' = \eta,\) and the mapping \( \omega' \mapsto \delta\omega' \) is \( C \)-to-\(1 \) for \(C  =| \{  \eta \in \Omega_1(B_N,\mathbb{Z}_{m\lor n}) \colon d\eta=0\} | .\) Using this, it follows that the previous equation is equal to
	\begin{align*}
		&
		%\sum_{\substack{\eta \in  \Omega_1(B_N,\mathbb{Z}_{m\vee n}) \\ \omega \in  \Omega_2(B_N, \mathbb{Z}_m)}}
		%\prod_{p \in C_2(B_N)^+} \bar \varphi_{\beta,m}(\omega(p)) \prod_{e \in C_1(B_N)^+} \bar \varphi_{\kappa,m\lor n}((\eta-\gamma)(e))  
		%\\&\qquad\qquad\quad\cdot \mathbf{1}(\eta +\delta \omega =0 \pmod m)\mathbf{1}(\delta \eta  = 0 \pmod n ).
		%
		%\\  & = 
		%| \{  \eta \in \Omega_1(B_N,\mathbb{Z}_{m\lor n}) \colon d\eta=0\} |
		%\\&\qquad\qquad\cdot
		%  
		%  
		C \sum_{\substack{\omega' \in  \Omega_2(B_N,\mathbb{Z}_{m\vee n}) \\ \omega \in  \Omega_2(B_N, \mathbb{Z}_m)}}
		\prod_{p \in C_2(B_N)^+} \bar \varphi_{\beta,m}(\omega(p)) \prod_{e \in C_1(B_N)^+} \bar \varphi_{\kappa,m\lor n}((\delta \omega'-\gamma)(e))  
		\\&\qquad\qquad\cdot \mathbf{1}(\delta(\omega' + \omega) =0 \pmod m) .
	\end{align*}
	From this, the desired conclusion immediately follows. 
\end{proof}

\subsubsection{The special case \( G=H=\mathbb{Z}_2 \)}\label{sec: high temperature Z2}

In the special case \( G=H=\mathbb{Z}_2, \) using the notation of Section~\ref{sec: high temperature ZmZn}, one has
\begin{align*}
		 \bar \varphi_{a,2}(0) =
		 %\sum_{i \in \mathbb{Z}_{n}} \varphi_{a,2}(i)\varphi_{a,2}(i) = \varphi_{a,2}(0)\varphi_{a,2}(0)+\varphi_{a,2}(1)\varphi_{a,2}(1) = 
		 \cosh(2a)
		 \quad \text{and} \quad
		 \bar \varphi_{a,2}(1)  = 
		 %\sum_{i \in \mathbb{Z}_{2}} \varphi_{a,2}(i)\varphi_{a,2}(i-1) =2\varphi_{a,2}(0)\varphi_{a,2}(1) = 
		 \sinh(2a).
\end{align*}
To simplify notation, for \( j \in \mathbb{Z}_2, \) we let 
\begin{equation*}
	\hat\varphi_a(j) \coloneqq \hat \varphi_{a,2}(j) \coloneqq \frac{\bar \varphi_{a,2}(j)}{\bar \varphi_{a,2}(0)}
\end{equation*}
so that \( \hat\varphi_a(0) = 1 \) and \( \hat\varphi_a(1) = \tanh (2a). \)
With this notation, by Proposition~\ref{proposition: high high mn}, for any path \( \gamma \in C^1(B_N,\mathbb{Z})\) we have
	\begin{equation*}\begin{split}
		Z[\gamma]  & \propto
		\sum_{\substack{\omega' \in  \Omega_2(B_N,\mathbb{Z}_{2}) \\ \omega \in  \Omega_2(B_N, \mathbb{Z}_2)}}
		\prod_{p \in C_2(B_N)^+} \bar \varphi_{\beta,2}(\omega(p)) \prod_{e \in C_1(B_N)^+} \bar \varphi_{\kappa,2}\bigl( \delta \omega'(e)-\gamma [e]\bigr)  
		\\&\qquad\qquad\cdot \mathbf{1}(\delta(\omega' + \omega) =0 \pmod 2) 
		\\
		& \propto
		\sum_{\substack{ \omega \in  \Omega_2(B_N, \mathbb{Z}_2)}}
		\prod_{p \in C_2(B_N)^+} \bar \varphi_{\beta,2}(\omega(p)) \prod_{e \in C_1(B_N)^+} \bar \varphi_{\kappa,2}\bigl(\delta \omega (e)-\gamma [e]\bigr)  
		\\
		& \propto
		\sum_{\substack{ \omega \in  \Omega_2(B_N, \mathbb{Z}_2)}}
		\prod_{p \in C_2(B_N)^+} \hat \varphi_{\beta}(\omega(p)) \prod_{e \in C_1(B_N)^+} \hat \varphi_{\kappa} \bigl(\delta \omega (e)-\gamma [e] \bigr)  
		\\
		& \propto
		\hat \varphi_{\kappa}(1)^{|\gamma|}\sum_{\substack{ \omega \in  \Omega_2(B_N, \mathbb{Z}_2)}}
		\prod_{p \in C_2(B_N)^+} \hat \varphi_{\beta}(\omega(p)) \prod_{e \in C_1(B_N)^+}  \frac{\hat \varphi_{\kappa}\bigl(\delta \omega (e)-\gamma [e]\bigr)}{\hat \varphi_{\kappa}(\gamma [e])}.
	\end{split}\end{equation*} 
	Above, the proportionality constants depends on \( N ,\) \( \beta, \) and \( \kappa, \) but not on \( \gamma. \)

To simplify  notation, we let 
\begin{equation}\label{eq: ht action}
    \varphi^\gamma_{\beta,\kappa} (\omega)  \coloneqq   
     \prod_{p \in C_2(B_N)^+} 
     \hat \varphi_\beta\bigl(\omega(p)\bigr) \prod_{e \in C_1(B_N)^+}  
     \frac{\hat \varphi_\kappa\bigl( \delta\omega(e) + \gamma[e] \bigr)}{\hat \varphi_\kappa(\gamma[e])} .
\end{equation} 
Further, we let
\begin{equation}\label{eq: hat Z confinement}
	\hat Z_{N,\beta,\kappa}[\gamma] \coloneqq \hat \varphi_\kappa\bigl( 1\bigr)^{|\gamma|} \sum_{\omega \in \Omega_2(B_N,\mathbb{Z}_2)} \varphi^\gamma_{\beta,\kappa}(\omega)   .
\end{equation}
For a path \( \gamma \in C^1(B_N,\mathbb{Z}),\) we also define 
\[
    \hat Z_{N,\beta,\kappa}^\gamma \coloneqq \sum_{\substack{\omega \in \Omega_2(B_N,\mathbb{Z}_2)  \mathrlap{\colon}\\  \not \exists p \in \support \omega \colon p \sim \gamma    }}\varphi^0_{\beta,\kappa}(\omega) .
\]

\subsection{Unitary gauge}\label{sec: unitary gauge} 
In this section, we recall the notion of gauge transforms and describe how these can be used to rewrite the Wilson line expectation as an expectation with respect to a slightly simpler probability measure in the special case \( G=H.\) For more details on gauge transforms and unitary gauge, we refer the reader to~\cite{f2022b}.

For \( \eta \in \Omega_0(B_N,G) \), consider the map
\[
    \tau \coloneqq \tau_\eta \coloneqq \tau_\eta^{(1)} \cdot \tau_\eta^{(2)} \colon \Omega_1(B_N,G)  \cdot \Omega_0(B_N,G)  \to  \Omega_1(B_N,G) \cdot \Omega_0(B_N,G) 
\]
given by
\begin{equation}\label{eq: gauge transform}
    \begin{cases}
     \sigma(e) \mapsto  \sigma(e) -d\eta(e), & e \in C_1(B_N), \cr
     \phi(x) \mapsto  \phi(x) + \eta(x), & x \in C_0(B_N),
    \end{cases}
\end{equation}
where \( \sigma \in \Omega_1(B_N,G) \) and \( \phi \in \Omega_0(B_N,G). \) 
A map \( \tau \) of this form is called a \emph{gauge transformation}. %, and functions \( f: \Omega_1(B_N,G) \cdot \Omega_0(B_N,G)  \to \mathbb{C} \) which are invariant under such mappings in the sense that \(f= f \circ \tau\) are said to be \emph{gauge invariant}.  
An important thing about gauge transformations is that for  \( \eta \in \Omega_0(B_n,H), \) by the Poincar\'e lemma (see, e.g., \cite{c2019}), the map \( \eta \mapsto  \sigma - d\eta \) is \( k \)-to-\( 1 \) for some \( k \geq 1. \)

The gauge transform \( \tau_\eta \) which will be most useful to us is the gauge transform corresponding to \( \eta = -\phi, \) which maps \( \sigma -d\phi \) to \( \sigma. \) After having applied this gauge transformation, we say we are working in \emph{unitary gauge}. 
For \( \beta, \kappa \geq   0 \) and  \( \sigma \in \Omega_1(B_N,G) \), we  define the probability measure corresponding to unitary gauge as
\begin{equation}\label{eq: fixed length unitary measure}
    \mu_{N,\beta,\kappa}^{(U)}(\sigma) 
    \coloneqq
    (Z_{N,\beta,\kappa}^{(U)})^{-1}\exp\pigl(\beta \sum_{p \in C_2(B_N)}      \rho\bigl(d  \sigma(p)\bigr) + \kappa \sum_{e \in C_1(B_N)}   \rho \bigl( \sigma(e)\bigr)\pigr)  ,
\end{equation} 
where \( Z_{N,\beta,\kappa}^{(U)} \) is a normalizing constant. We let \( \mathbb{E}^{(U)}_{N,\beta,\kappa} \) denote the corresponding expectation.

The following lemma, which is considered well-known in the physics literature, will be crucial in the analysis of the lattice Higgs model in the case \( G = H , \) as it will allow us to consider a model with only one field, which will in addition form finite edge clusters whenever \( \kappa \) is sufficiently large, 

\begin{lemma}[Corollary 2.17 in~\cite{f2022b}]\label{lemma: unitary gauge}
    Let \( \beta \geq 0,\) \( \kappa \geq 0 \), \( G = H , \) and let \( \gamma \in  C_1(B_N) \) be a path. Then
    \begin{equation*}
        \mathbb{E}_{N,\beta,\kappa}\bigl[W_\gamma(\sigma,\phi)\bigr] =
        \mathbb{E}^{(U)}_{N,\beta,\kappa}\bigl[W_\gamma(\sigma,1)\bigr] =
        \mathbb{E}^{(U)}_{N,\beta,\kappa}\pigl[\rho\bigl(\sigma(\gamma)\bigr)\pigr]. 
    \end{equation*} 
\end{lemma}

In~\cite{f2022b}, Lemma~\ref{lemma: unitary gauge} was proven only in the case when \( \gamma \) was a path, but it holds (with the same proof) more generally also for 1-chains in the sense that for any 1-chain \( c \in C_1(B_N), \) one has 
    \begin{equation*}
        \mathbb{E}_{N,\beta,\kappa}\pigl[ \rho \bigl( \sigma(c) -\phi(\partial c)\bigr) \pigr] = 
        \mathbb{E}^{(U)}_{N,\beta,\kappa}\pigl[\rho\bigl(\sigma(c)\bigr)\pigr]. 
    \end{equation*}

With the current section in mind, we will work with \( \sigma \sim \mu_{N,\beta, \kappa}^{(U)} \) rather than with \( {(\sigma,\phi) \sim \mu_{N,\beta,\kappa}}\) throughout the rest of this paper, together with the observable
\begin{equation*}
    W_\gamma(\sigma) \coloneqq W_\gamma(\sigma,1) = \prod_{e \in \gamma} \rho\bigl(\sigma(e)\bigr) = \rho(\sigma(\gamma)).
\end{equation*}  
To simplify notation, we further let
\begin{equation*}
	 Z_{N,\beta,\kappa}^{(U)}[\gamma] = \sum_{\sigma \in \Omega_1(B_N,G)} \rho\bigl(\sigma(\gamma)\bigr) \exp\pigl(\beta \sum_{p \in C_2(B_N)}      \rho\bigl(d  \sigma(p)\bigr) + \kappa \sum_{e \in C_1(B_N)}   \rho \bigl( \sigma(e)\bigr)\pigr)  
\end{equation*}
and note that with this notation, \( Z_{N,\beta,\kappa}^{(U)} = Z_{N,\beta,\kappa}^{(U)}[0]. \)

\subsection{The activity}\label{sec: activity}

For \( a\geq 0 \) and \( g \in G ,\) we set
\begin{equation*}
    \psi_{a}(g) \coloneqq  e^{a \real  (\rho(g)-\rho(0))} .
\end{equation*} 
Since \( \rho \) is unitary, for any \( g \in G \) we have \( \rho(g) = \overline{\rho(-g)} \), and hence \( \real  \rho(g) = \real  \rho(-g) \). In particular,  for any \( g \in G \) 
\begin{equation} \label{eq: phi is symmetric}
    \psi_{a}(g)
    =
    e^{ a (\real  \rho(g)-\rho(0))  }
    =
    e^{a (\real  \rho(-g)-\rho(0)) }
    =
    \psi_{a}(-g).
\end{equation}
For \( \sigma \in \Omega_1(B_N,G) \) and \( \beta,\kappa \geq 0 \) we define the \emph{activity} of $\sigma$ by
\begin{equation*}
    \psi_{\beta,\kappa}(\sigma) \coloneqq \prod_{p \in C_2(B_N)}\psi_\beta\bigl(d\sigma(p)\bigr) \prod_{e \in C_1(B_N)}\psi_\kappa\bigl(\sigma(e)\bigr).
\end{equation*}
We note that, by definition, if \( \sigma,\sigma' \in \Omega_1(B_N,G),\) \( \support \sigma' \subseteq \support \sigma, \) \( \sigma|_{\support \sigma'} = \sigma' \)  and  \( (d\sigma)|_{\support d\sigma'} = d\sigma',\) then \[ \psi_{\beta,\kappa} (\sigma) = \psi_{\beta,\kappa}(\sigma') \psi_{\beta,\kappa}(\sigma-\sigma').\]

For \( \sigma \in \Omega_1(B_N,G), \) the probability measure corresponding to (Wilson action) lattice gauge theory in unitary gauge can be written as
\begin{equation}\label{eq: mubetakappaphi}
    \mu^{(U)}_{N,\beta,\kappa}(\sigma) = \frac{\psi_{\beta,\kappa} (\sigma)}{\sum_{\sigma \in \Omega_1(B_N,G)} \psi_{\beta,\kappa}(\sigma)}.
\end{equation}

In the case  $G = \mathbb{Z}_2,$ for \( \sigma \in \Omega_1(B_N,\mathbb{Z}_2)\) we have
\begin{equation*}
    \begin{split}
        &\psi_{\beta,\kappa}(\sigma) = \prod_{p \in C_2(B_N)} \psi_{\beta}\bigl( d\sigma(p) \bigr)\prod_{e \in C_1(B_N)} \psi_{\kappa}\bigl( \sigma(e) \bigr) 
        \\&\qquad= \prod_{p \in C_2(B_N)} e^{-2\beta \mathbb{1}\bigl( d\sigma(p) = 1\bigr)}\prod_{e \in C_1(B_N)} e^{-2\kappa \mathbb{1}( \sigma(e) = 1)}
        \\&\qquad=
        e^{-2\beta \sum_{p \in C_2(B_N)} \mathbb{1}( d\sigma(p) = 1)}
        e^{-2\kappa \sum_{e \in C_1(B_N)} \mathbb{1}( \sigma(e) = 1)}
        \\&\qquad=
        e^{-2\beta |\support d\sigma|}e^{-2\kappa |\support \sigma|}.
    \end{split}
\end{equation*}

\subsection{Cluster expansions}\label{section: cluster expansions}

In this section, we will describe the cluster expansions which we will use in the proofs of our main results. In both the Higgs phase and confinement regime, these are the same expansions that were used in~\cite{f2024}. However, here we improve several of the technical lemmas from~\cite{f2024} for them to apply also to the loops and paths considered here. 

\subsubsection{The Higgs phase}\label{section: cluster expansions in Higgs phase}

In this section, we recall the cluster expansion for the lattice Higgs model on a finite box $B_N$ in the Higgs phase as described in~\cite{f2024}, which is a special case of the cluster expansion described in~\cite{fv2017}.

Throughout this section, we will assume that \( G = \mathbb{Z}_2.\) However, we mention that the same ideas can be used, with some work, for any finite abelian group.

\paragraph{Polymers}

Let  \( \mathcal{G}_1 \) be the graph with vertex set \( C_1(B_N)^+\) and an edge between two distinct vertices $e_1,e_2$ iff  $\support \hat \partial e_1 \cap \support \hat \partial e_2  \neq \emptyset$, i.e.{} if \( e_1\) and \( \pm e_2\) are both in boundary of some common plaquette. 
Since any edge \( e \in C_1(B_N)^+\) in \( B_N\) is in the boundary of at most \( 2(d-1)\) plaquettes, and any such plaquettes has exactly three edges in its boundary that are not equal to \( e,\) it follows that there are at most 
	\begin{equation}\label{eq: M1}
		M_1 \coloneqq 3 \cdot 2(d-1) = 6(d-1)
	\end{equation} 
	edges \( e' \in C_1(B_N)^+\smallsetminus \{ e \} \) with \( {\support \hat \partial e \cap \support \hat \partial e' \neq \emptyset.}\) As a consequence, it follows that each vertex in \( \mathcal{G}_1 \) has degree at most \( M_1.\) 
The graph \( \mathcal{G}_1\) will be useful when we describe the cluster expansion when we use it in the Higgs phase.  

When \( \sigma \in \Omega_1(B_N,G),\) we let \( \mathcal{G}_1 (\sigma ) \) be the subgraph of \( \mathcal{G}_1 \) induced by \( (\support \sigma)^+. \)
We let \( \Lambda = \Lambda_N \) be the set of all \( \sigma \in \Omega_1(B_N,G)\) such that  \( \mathcal{G}_1 (\sigma )  \) has exactly one connected component. 
The spin configurations in \( \Lambda \) will be referred to as \emph{polymers}.

For \( \sigma,\sigma' \in \Lambda,\) we write \( \sigma \sim \sigma'\) if   \( \mathcal{G}_1(\sigma) \cup \mathcal{G}_1(\sigma')\) is a connected subgraph of \( \mathcal{G}_1.\)
In the notation of~\cite[Chapter 5]{fv2017}, the model given by~\eqref{eq: hat Z confinement} corresponds to a model of polymers with polymers as described above and interaction function 
\begin{equation*}
	\iota(\sigma_1,\sigma_2) \coloneqq 
	\begin{cases}  
		0 &\text{if }  \sigma_1 \sim \sigma_2 \cr   
		1 &\text{else} 
	\end{cases}\qquad  \sigma_1,\sigma_2 \in \Lambda.
\end{equation*}

\paragraph{Clusters of polymers}

Consider a multiset 
\begin{align*}
    &\mathcal{S} = \{ \underbrace{\eta_1,\dots, \eta_1}_{n_{\mathcal{S}}(\eta_1) \text{ cdot}}, \underbrace{\eta_2, \dots, \eta_2}_{n_{\mathcal{S}}(\eta_2) \text{ cdot}},\dots, \underbrace{\eta_k,\dots,\eta_k}_{n_{\mathcal{S}}(\eta_k) \text{ cdot}} \} = \{\eta_1^{n(\eta_1)}, \ldots, \eta_k^{n(\eta_k)}\}, 
\end{align*}
where \( \eta_1,\dots,\eta_k \in \Lambda  \) are distinct and $n(\eta)=n_{\mathcal{S}}(\eta)$ denotes the number of cdot $\eta$ occurs in~$\mathcal{S}$. Following \cite[Chapter 3]{fv2017}, we say that $\mathcal{S}$ is \emph{decomposable} if it is possible to partition $\mathcal{S}$ into disjoint multisets. That is, if there exist non-empty and disjoint multisets $\mathcal{S}_1,\mathcal{S}_2 \subset \mathcal{S}$ such that $\mathcal{S} = \mathcal{S}_1 \cup \mathcal{S}_2$ and such that for each pair $(\eta_1,\eta_2) \in \mathcal{S}_1 \cdot \mathcal{S}_2$,  $ \eta_1 \nsim \eta_2.$  
If $\mathcal{S}$ is not decomposable, we say that $\mathcal{S}$  is a \emph{cluster}. We stress that a cluster is unordered and may contain several copies of the same polymer. Given a cluster $\mathcal{S} $, we define
\begin{equation*}
    \|\mathcal{S} \|_1 = \sum_{\eta \in \Lambda } n_{\mathcal{S}} (\eta) \bigl| (\support \eta )^+ \bigr|,\quad 
    \|\mathcal{S} \|_2 = \sum_{\eta \in \Lambda } n_{\mathcal{S}} (\eta) \bigl| (\support d\eta )^+ \bigr| ,
\end{equation*}
\begin{equation*}
    n(\mathcal{S}) = \sum_{\eta \in \Lambda } n_{\mathcal{S}} (\eta),\quad \text{and} \quad
    \support \mathcal{S} = \bigcup_{\eta \in \mathcal{S}} \support \eta.
\end{equation*}
For a $1-$chain \( c \in C^1(B_N,\mathbb{Z}),\) we define 
\[ \mathcal{S}(c) = \sum_{\eta \in \Lambda } n_{\mathcal{S}}(\eta) \eta(c).\]
%Given $X$ and $\nu \in X$, we write $n_X(\nu)$ for the number of occurrences of $\nu$ in $X$.

We let \( \Xi  \) be the set of all clusters.  
For \( e \in C_1(B_N)^+,\) \( \gamma \in C^1(B_N,\mathbb{Z}), \) and \( i ,j,k \geq 1 ,\) we define
\begin{equation*}
    \Xi_{i,j,k,e} \coloneqq \bigl\{ \mathcal{S} \in \Xi \colon 
    n(\mathcal{S}) = i,\, \| \mathcal{S}\|_1 = j,\, \| \mathcal{S}\|_2 = k,\,   e \in \support \mathcal{S} \bigr\},
\end{equation*}  
\begin{equation*}
    \Xi_{i} \coloneqq \bigl\{ \mathcal{S} \in \Xi \colon 
    n(\mathcal{S}) = i \bigr\},
\end{equation*} 
\begin{equation*}
    \Xi_{i,j} \coloneqq \bigl\{ \mathcal{S} \in \Xi \colon 
    n(\mathcal{S}) = i,\, \|\mathcal{S}\|_1 = j \bigr\},
\end{equation*} 
\begin{equation*}
    \Xi_{i,j,e} \coloneqq \bigl\{ \mathcal{S} \in \Xi \colon 
    n(\mathcal{S}) = i,\, \|\mathcal{S}\|_1 = j, \, e \in \support \mathcal{S} \bigr\},
\end{equation*} 
and
\begin{equation*}
    \Xi_{i,j,\gamma} \coloneqq \bigl\{ \mathcal{S} \in \Xi \colon 
    n(\mathcal{S}) = i,\, \|\mathcal{S}\|_1 = j, \, \mathcal{S}(\gamma) \neq 0 \bigr\}.
\end{equation*} 
Further, we let
\begin{equation*}
    \Xi_{i^+ } \coloneqq \bigl\{ \mathcal{S} \in \Xi \colon 
    n(\mathcal{S}) \geq i\},
\end{equation*} 
and define \( \Xi_{i^+,j}, \) \( \Xi_{i,j^+}, \) \( \Xi_{i^+,j^+}, \) etc. analogously. 
We note that the sets defined above depend on \( N\) but we usually suppress this in the notation. When we want to emphasize this dependence, we write \( \Xi^{(N)}, \) \( \Xi_i^{(N)},\) \( \Xi_{i,j}^{(N)}, \) etc.

 We extend the notion of activity from Section~\ref{sec: activity} to clusters \( \mathcal{S} \in \Xi\) by letting
\begin{equation*}
    \psi_{\beta,\kappa}(\mathcal{S}) \coloneqq \prod_{\eta \in \mathcal{S}} \psi_{\beta,\kappa}(\eta)^{n_{\mathcal{S}}(\eta)} .
\end{equation*}
When \( G=H = \mathbb{Z}_2, \) we note that
\begin{equation*}
    \psi_{\beta,\kappa}(\mathcal{S}) = e^{-4\beta \| \mathcal{S}\|_2-4\kappa \| \mathcal{S}\|_1}.
\end{equation*}

 We note that we get \( 4 \beta \) and \( 4\kappa \) here as \( \| \mathcal{S}\|_1 \) and \( \|\mathcal{S}\|_2 \) counts the number of positively oriented edges and plaquettes in the support of \( \mathcal{S} \) with multiplicities instead of all edges and plaquettes in the support.

\paragraph{Ursell functions} 
In cluster expansions, special functions known as Ursell functions naturally appear. We define the Ursell function relevant to our setting below.

For \( k \geq 1,\) we let \( \mathcal{G}^k\) be the set of all connected graphs \( G\) with vertex set \(  \{ 1,2,\dots, k\}.\) Whenever \( G\) is a graph, we let \( E(G)\) be the (undirected) edge set of \( G.\)

\begin{definition}[The Ursell functions]\label{def:ursell}
    For \( k \geq 1 \) and  \( \eta_1,\eta_2,\dots , \eta_k \in \Lambda ,\) we let
    \begin{equation*}
        U(  \eta_1, \ldots, \eta_k  ) \coloneqq \frac{1}{k!} \sum_{G \in \mathcal{G}^k} (-1)^{|E(\mathcal{G} )|} \prod_{(i,j) \in E(\mathcal{G})} \mathbb{1}( \eta_i \sim \eta_j).
    \end{equation*} 
    Note that this definition is invariant under permutations of the polymers \( \eta_1,\eta_2,\dots,\eta_k.\)

   For \( \mathcal{S} \in \Xi_k,\) and any enumeration \( \eta_1,\dots, \eta_k \) (with multiplicities) of the polymers in \( \mathcal{S},\) we define
   \begin{equation}\label{eq: ursell functions} 
        U(\mathcal{S}) = k! \, U(\eta_1,\dots, \eta_k).
   \end{equation}
\end{definition} 
Note that for any \( \mathcal{S} \in \Xi_1,\) we have \( U(\mathcal{S})=1,\) and for any \( \mathcal{S} \in \Xi_2,\)  we have
\( U(\mathcal{S}) = -1. \)

\paragraph{Cluster expansion of the partition function and Wilson line observables}
The partition function for the Ising lattice Higgs model with parameters \( \beta,\kappa \geq 0 \) is given by
\[
Z^{(U)}_{N,\beta,\kappa} =  \sum_{\sigma \in \Omega_1(B_N,G)} \psi_{\beta,\kappa}(\sigma). 
\]
Since \( N \) is finite, this is a finite sum.
Now instead consider the \emph{cluster partition function}, which is defined by the following (formal) expression:
\begin{equation}\label{eq: spin-cluster-partition}
     Z_{N,\beta,\kappa}^{*} = \exp\Bigl(
    \sum_{\mathcal{S}\in \Xi} \Psi_{\beta,\kappa}(\mathcal{S})\Bigr), 
\end{equation}
where for \( \mathcal{S} \in \Xi, \) we define
\begin{equation*}
    \begin{split}
        \Psi_{\beta,\kappa}(\mathcal{S}) \coloneqq U(\mathcal{S})
        \psi_{\beta,\kappa}(\mathcal{S}) 
    \end{split}
\end{equation*}
and $U$ is the Ursell function as defined in~\eqref{eq: ursell functions}.
Recalling the definition of \( M_1 \) from~\eqref{eq: M1}, we define
\begin{equation}\label{eq: kappa0} 
    \alpha = \alpha^{(\textrm{Higgs})}\coloneqq             \argmin_{\alpha'\in (0,1)}\frac{\log (M_1^2+1/\alpha' )}{4(1-\alpha')} 
    \quad \text{and} \quad 
    \kappa_0^{(\textrm{Higgs})} \coloneqq \frac{\log (M_1^2+1/\alpha )}{4(1-\alpha)}.
\end{equation}

By~\cite[Lemma~3.3]{f2024}, the series on the right-hand side of~\eqref{eq: spin-cluster-partition} is absolutely convergent whenever \( \kappa > \kappa_0^{\mathrm{(Higgs)}} , \) and by~\cite[Lemma 3.4]{f2024}, in this case we further have 
\begin{equation}\label{eq: log Z expansion}
    \log Z^{(U)}_{N,\beta,\kappa} = \log Z_{N,\beta,\kappa}^{*} = \sum_{\mathcal{S} \in \Xi} \Psi_{\beta,\kappa}(\mathcal{S}).
\end{equation}
In~\cite[Lemma~3.3]{f2024}, there is an additional assumption that the path \( \gamma \) had a specific shape, but since this is not used in its proof, it holds more generally for any path \( \gamma \in C^1(B_N,\mathbb{Z}) .\)

Given a path \( \gamma \in C1(B_N,\mathbb{Z}), \) consider the weighted cluster partition function %{eq: spin-cluster-partition}{eq:dec5.1}
\begin{equation}\label{eq:dec5.1}
   Z_{N,\beta,\kappa}^{*}[\gamma] \coloneqq \exp\Bigl(\sum_{\mathcal{S} \in \Xi} \Psi_{\beta,\kappa,\gamma}(\mathcal{S})\Bigr), 
\end{equation}
where
\[
    \Psi_{\beta,\kappa,\gamma}(\mathcal{S}) \coloneqq U(\mathcal{S}) \psi_{\beta,\kappa}(\mathcal{S})  
    \rho\bigl( \mathcal{S}(\gamma) \bigr). 
\]
By~\cite[Section 3.1.7]{f2024}, when $\beta\geq 0$ and \( \kappa  \geq \kappa_0^{(\textrm{Higgs})}, \) the series on the right-hand side of~\eqref{eq:dec5.1} is absolutely convergent, and moreover,
\begin{equation}\label{eq: log ZUgamma}
    \log Z^{(U)}_{\beta,\kappa,N}[\gamma] = \log Z_{N,\beta,\kappa}^{*}[\gamma]  = \sum_{\mathcal{S} \in \Xi} \Psi_{\beta,\kappa,\gamma}(\mathcal{S}).
\end{equation}

\paragraph{Useful upper bounds}

In this section, we recall a lemma from~\cite{f2024}, which gives an upper bound on sums over clusters.

\begin{lemma}[Lemma 3.10 in~\cite{f2024}]\label{lemma: new upper bound}
    Let \( \beta \geq 0 \) and \( \kappa > \kappa_0^{(\textrm{Higgs})}.\) Further, let \( k \geq 1\) and \( e \in C_1(B_N).\)  
    Then, for any \( \varepsilon>0,\) we have
    \begin{equation*}
        \begin{split}
            & \sum_{\mathcal{S} \in \Xi_{1+,k+,e}} \bigl| \Psi_{\beta,\kappa}(\mathcal{S}) \bigr|   
            \leq  
            4^{-1}C_{\varepsilon}e^{-4k(\kappa-\kappa_0^{(\textrm{Higgs})}-\varepsilon) },
        \end{split}
    \end{equation*} 
    where \( C_{\varepsilon}\) is defined by 
    \begin{equation*}\label{eq: cbeta*}
        C_{\varepsilon} \coloneqq 4\sup_{N \geq 1} \sup_{e \in C_1(B_N)}\sum_{\mathcal{S} \in \Xi_{1+,1+,e}} \bigl| \Psi_{0,\kappa_0^{(\textrm{Higgs})}+\varepsilon}(\mathcal{S}) \bigr| <\infty.
        %\frac{ 4e^{-4\varepsilon-2(2\kappa_0^{\textrm{Higgs})}-\alpha) }}{ 1-4M_1^2 e^{-4\varepsilon-2(2\kappa_0^{(\textrm{Higgs})} -\alpha) }}.
    \end{equation*} 
\end{lemma}

\subsubsection{The confinement phase}\label{section: cluster expansions in confinement phase}

In this section, we describe the cluster expansion we will use in the confinement phase. The main difference between the approaches we use in the Higgs phase and the confinement phase is that in the latter, we first perform a high-temperature expansion (see Section~\ref{sec: high temperature Z2}) and then use a cluster expansion on the model obtained thereof.

\paragraph{Polymers}

In this section, we consider the case \( G=H=\mathbb{Z}_2, \) and assume that a path  \( \gamma \in C^1(B_N,\mathbb{Z}) \) is given.

Let \( \mathcal{G}_2\) be the graph with vertex set \(  C_2(B_N)^+,\) an edge between two distinct plaquettes \( p_1,p_2 \in C_2(B_N)^+ \) if \( \support  \partial p \cap \support  \partial p' \neq \emptyset\) (written \( p_1 \sim p_2\)). In other words, \( p_1 \sim p_2 \) if they have some common edge in their boundaries.
Note that any \( p \in C_2(B_N)^+ \) has degree at most \( M_2 \coloneqq 4 \cdot \bigl( 2(d-1)-1 \bigr) =    8d-12\) in \( \mathcal{G}_2.\)

For \( \omega \in \Omega_2(B_N,\mathbb{Z}_2), \) we let \( \mathcal{G}_2(\omega)\) be the subgraph of \( \mathcal{G}_2 \) induced by \(  (\support \omega)^+.\) In other words, \( \mathcal{G}_2 \) is the graph with vertex set  \(  (\support \omega)^+\) and an edge between two vertices \( p_1,p_2 \in (\support \omega)^+ \) if \( p_1 \sim p_2. \)

In this section, we let
\begin{equation*}
    \Lambda \coloneqq \bigl\{ \omega \in \Omega_2(B_N,\mathbb{Z}_2) \colon  \mathcal{G}_2(\omega) \text{ has exactly one connected component} \bigr\}.
\end{equation*} 
The spin configurations in \( \Lambda  \) will be referred to as \emph{polymers}.

For \( \omega,\omega' \in \Lambda,\) we write \( \omega \sim \omega'\) if   \( \mathcal{G}_2(\omega) \cup \mathcal{G}_2(\omega')\) is a connected subgraph of \( \mathcal{G}_2.\)

In the notation of~\cite[Chapter 3]{fv2017}, the model given by~\eqref{eq: hat Z confinement} corresponds to a model of polymers with polymers described in the previous section and interaction function 
\begin{equation*}
	\iota(\omega_1,\omega_2) \coloneqq  
	\begin{cases}  
		0 &\text{if }  \omega_1 \sim \omega_2 \cr   
		1 &\text{else,} 
	\end{cases}
	\qquad \omega_1,\omega_2 \in \Lambda .
\end{equation*}

For \( \omega \in \Omega_2(B_N,\mathbb{Z}_2),\) we write \( \omega \sim \gamma\) if there is \( p \in \support \omega\) such that \( \support \partial p \cap \support  \gamma \neq \emptyset.\) Finally, we also write \( \gamma \sim \gamma.\)

\paragraph{Clusters of polymers}

Exactly as in the Higgs phase, consider the multiset 
\begin{align*}
    &\mathcal{S} = \{ \underbrace{\eta_1,\dots, \eta_1}_{n_{\mathcal{S}}(\eta_1) \text{ cdot}}, \underbrace{\eta_2, \dots, \eta_2}_{n_{\mathcal{S}}(\eta_2) \text{ cdot}},\dots, \underbrace{\eta_k,\dots,\eta_k}_{n_{\mathcal{S}}(\eta_k) \text{ cdot}} \} = \{\eta_1^{n(\eta_1)}, \ldots, \eta_k^{n(\eta_k)}\}, 
\end{align*}
where \( \eta_1,\dots,\eta_k \in \Lambda \) are distinct and $n(\eta)=n_{\mathcal{S}}(\eta)$ denotes the number of cdot $\eta$ occurs in~$\mathcal{S}$. Following \cite[Chapter 3]{fv2017}, we say that $\mathcal{S}$ is \emph{decomposable} if it is possible to partition $\mathcal{S}$ into disjoint multisets. That is, if there exist non-empty and disjoint multisets $\mathcal{S}_1,\mathcal{S}_2 \subset \mathcal{S}$ such that $\mathcal{S} = \mathcal{S}_1 \cup \mathcal{S}_2$ and such that for each pair $(\eta_1,\eta_2) \in \mathcal{S}_1 \cdot \mathcal{S}_2$,  $\eta_1 \nsim \eta_2.$  
If $\mathcal{S}$ is not decomposable, we say that $\mathcal{S}$  is a \emph{cluster}. We stress that a cluster is unordered and may contain several copies of the same polymer. 
We let \(  \Xi\) be the set of all clusters.

For \( \mathcal{S} \in  \Xi,\) we let \( \mathcal{S}^0\) be the set of all  \(\eta\in  \mathcal{S}\) such that \( \eta \nsim \gamma .\)  We let \( \mathcal{S}^\gamma \coloneqq \mathcal{S}\smallsetminus \mathcal{S}^0.\)

Given a cluster $\mathcal{S} \in  \Xi$, we let
\begin{equation*}
    \|\mathcal{S} \| = \sum_{\eta \in \Lambda  } n_{\mathcal{S}} (\eta) \bigl| (\support \eta )^+ \bigr| ,
\end{equation*} 
\begin{equation*}
    \|\mathcal{S} \|_1 \coloneqq \sum_{\eta \in \Lambda  } n_{\mathcal{S}^0} (\eta) \bigl| (\support \delta \eta   )^+ \bigr|  ,
\end{equation*}
and
\begin{equation*}
    \|\mathcal{S} \|_\gamma \coloneqq \sum_{\eta \in \Lambda  } n_{\mathcal{S}^\gamma} (\eta) \bigl| \support \delta \eta \cap \support \gamma \bigr|  .
\end{equation*}
We also let
\begin{equation*}
    \support \mathcal{S}   = \bigcup_{\eta \in \mathcal{S}   }   (\support \eta )^+ .
\end{equation*} 
Further, for \( p \in C_2(B_N)^+ \) and \( j \geq 1, \) we let
\begin{equation*}
    \Xi_{1,j,p}   \coloneqq \bigl\{\{ \eta \} \colon \eta \in \Lambda  \text{ and } p \in \support \eta \text{ and } \| \{ \eta\} \| = j\bigr\}
\end{equation*}  
and
\begin{equation*}
    \Xi_{1+,j+,p}   \coloneqq \bigl\{\mathcal{S} \in \Xi \colon   p \in \support \mathcal{S} \text{ and } \| S\| \geq j\bigr\}.
\end{equation*}  
These sets depends on \( N,\) and when this dependency is important, we write \( \Xi^{(N)}, \) \(\Xi_{1,j,p}^{(n)}, \) and \(\Xi_{1+,j+,p}^{(N)} \) respectively.

\color{black}

\paragraph{The activity of clusters}

We extend the notion of activity from~\eqref{eq: ht action} to clusters \( \mathcal{S} \in \Xi\) as follows. 
\begin{align*}
    \varphi^\gamma_{\beta,\kappa}(\mathcal{S}) \coloneqq \prod_{\eta \in \mathcal{S}} \varphi^\gamma_{\beta,\kappa}(\eta).
\end{align*}
Note that 
\[
\varphi^\gamma_{\beta,\kappa}(\mathcal{S}) = \varphi^0_{\beta,\kappa}(\mathcal{\mathcal{S}}) \varphi_\kappa(1)^{-\| \mathcal{S}\|_\gamma}.
\]

The following lemma, which we now state and prove, generalizes~\cite[Lemma 4.3]{f2024}.

\begin{lemma}\label{lemma: geometry upper bound for paths}
	Let \( \beta,\kappa > 0. \) Further, let \( \gamma \in C^1(B_N,\mathbb{Z}) \) be a path along the boundary of a rectangle with side lengths \( R \) and \( T. \) Then, as \( R,T \to \infty, \) for any \( \omega \in \Lambda, \) we have
	\[
		\varphi^\gamma_{\beta,\kappa}(\omega) \leq \varphi_\beta(1)^{(1-o_{\min (R,T)}(1))|(\support \omega)^+|},
	\]
	where the constant only depends on \( \min (R,T). \)
	In particular, if \( \gamma \in C1(B_N,\mathbb{Z}) \) is a straight path and \( \omega \in \Lambda, \) then
	\begin{equation}\label{eq: geometry upper bound for paths}
		\varphi^\gamma_{\beta,\kappa}(\omega) \leq \varphi_\beta(1)^{|(\support \omega)^+|}.
	\end{equation}
\end{lemma}

\begin{proof}

	Let \( \gamma_{R,T} \) be a path around a rectangle with side lengths \( R \) and \( T \) such that \( \support \gamma \subseteq \support \gamma_{R,T}. \) Then \(
	\varphi^{\gamma}_{\beta,\kappa}(\omega) \leq  \varphi^{\gamma_{R,T}}_{\beta,\kappa}(\omega) . \) 
	The desired conclusion thus follows immediately from~\cite[Lemma 4.3]{f2024}. Moreover, by noting that if \( \gamma \) is a straight path, then we can take \( R \) and \( T \) arbitrarily large, we obtain~\eqref{eq: geometry upper bound for paths}. This concludes the proof.
\end{proof}

\paragraph{Ursell functions}

The Ursell functions we will use in the confinement phase are identical to the Ursell functions used for the Higgs phase (see Definition~\ref{def:ursell}), except the set \( \Lambda \) is different and we use the graph \( \mathcal{G}_2 \) instead of \( \mathcal{G}_1.\)

\paragraph{Cluster expansion of the partition function}

For \( \beta \geq 0, \) \( \kappa \geq 0 ,\) \( \gamma \in \Lambda, \) and \( \mathcal{S} \in \Xi, \) we let 
\begin{equation}\label{eq: Psi confinement}
	\Psi^\gamma_{\beta,\kappa}(\mathcal{S}) \coloneqq U(\mathcal{S}) \varphi^\gamma_{\beta,\kappa}(\mathcal{S}).
\end{equation} 
Further, we let
\begin{equation}\label{eq: beta0 def}
	\beta_0^{(\text{conf})} \coloneqq \sup \bigl\{ \beta \geq 0 \colon M_2^2 \varphi_\beta(1)<1  \text{ and } \exists \alpha \in (0,1) \text{ s.t. } \frac{M_2^3 \varphi_\beta(1)^{1-\alpha}}{1-M_2^2 \varphi_\beta(1)^{1-\alpha}}< 2\alpha \bigr\}  
\end{equation}
and for \( \beta < \beta_0^{(\text{conf})}, \) we let
\begin{equation}\label{eq: alpha cp}
	\alpha \coloneqq \alpha_{\beta} \coloneqq \inf \{ \alpha \in (0,1) \colon \frac{M_2^3 \varphi_\beta(1)^{1-\alpha}}{1-M_2^2 \varphi_\beta(1)^{1-\alpha}}< 2\alpha \}.
\end{equation}

Next, we present a lemma that shows that the cluster expansion of the logarithm of the partition function indeed converges. This lemma is a slightly more general version of~\cite[Lemma~4.4]{f2024}, and its proof is identical to the proof of~\cite[Lemma~4.4]{f2024} after replacing the reference to~\cite[Lemma~4.3]{f2024} in that proof with Lemma~\ref{lemma: geometry upper bound for paths} above.

\begin{lemma}\label{lemma: upper bound and convergence}
    Let \( 0 < \beta < \beta_0^{(\text{conf})} \) and \(  \kappa > 0,\)  and let \( \gamma \) be either a straight path or a path with support in the support of the boundary of some sufficiently large rectangle. Then, for any \( \omega \in \Xi,\) we have 
    \begin{equation*}
        \sum_{\mathcal{S} \in \Xi \colon \omega \in \mathcal{S}} \bigl|\Psi^\gamma_{\beta,\kappa}(\mathcal{S})\bigr| \leq \varphi^\gamma_{\beta,\kappa}(\omega)^{1-\alpha}  .
    \end{equation*} 
    Moreover, 
    \begin{equation}\label{eq: logZ}
        \log \hat Z_{N,\beta,\kappa} [\gamma]= \sum_{\mathcal{S} \in \Xi^\gamma} \Psi^\gamma_{\beta,\kappa}(\mathcal{S})
    \end{equation}
    and
    \begin{equation}\label{eq: logZgamma}
        \log \hat Z_{N,\beta,\kappa}^\gamma = \sum_{\mathcal{S} \in \Xi \colon \mathcal{S}^\gamma = \emptyset} \Psi^0_{\beta,\kappa}(\mathcal{S}).
    \end{equation}
    Furthermore, the series on the right-hand sides of~\eqref{eq: logZ} and~\eqref{eq: logZgamma} are both absolutely convergent.
\end{lemma}

\paragraph{Upper bounds for clusters}

In this paragraph, we give upper bounds on sums over the activity of sets of clusters, which naturally arise in the proofs of our main results in the confinement regime.

\begin{lemma}\label{lemma: general Cbeta cp}
    Let \( 0 < \beta < \beta_0^{(\text{conf})} \) and \(  \kappa >0 , \) let \( \gamma \in C1(B_N,\mathbb{Z}) \) either be a path whose support is a subset of the boundary of a rectangle with sufficiently large side lengths. Finally,  let \( p \in C_2(B_N)^+.\) Then  
    \begin{equation}\label{eq: cbeta* cp}
        \sum_{\mathcal{S} \in \Xi_{1^+,1^+,p}}
        \bigl| \Psi^\gamma_{\beta,\kappa}(\mathcal{S}) \bigr|
        \leq 
        \frac{ 
            \varphi_{\beta}(1)^{(1-o_n(1))(1-\alpha)}}{1-M_2^{2}
            \varphi_{\beta}(1)^{(1-o_n(1))(1-\alpha)}} \eqqcolon C_{\varepsilon,n}^{(2)}
    \end{equation}
    where \( \alpha\) is as in~\eqref{eq: alpha cp}.
\end{lemma}

The proof of this lemma is identical to the proof of~\cite[Lemma~4.5]{f2024}, except that the reference to~\cite[Lemma~4.3]{f2024} is replaced with a reference to the more general Lemma~\ref{lemma: geometry upper bound for paths}.

The next lemma, Lemma~\ref{lemma: new upper bound cp} below, gives an upper bound on the weight of all clusters which has support at a given plaquette, and in addition, is sufficiently large. This lemma will be crucial in the proofs of our main results to show that the error terms we get are small.
The proof of this lemma is identical to the proof of~\cite[Lemma 4.6]{f2024}, after replacing the reference therein to~\cite[Lemma 4.5]{f2024} with a reference to Lemma~\ref{lemma: general Cbeta cp}.
\begin{lemma}\label{lemma: new upper bound cp}
    Let \( 0 < \beta < \beta_0^{(\text{conf})} \) and \(  \kappa >0, \)  let \( \gamma \in C^1(B_N,\mathbb{Z}) \) be a path whose support is a subset of the boundary of a rectangle with sufficiently large side lengths, and let \( p \in C_2(B_N)^+.\) Further, let \( k \geq 1\) and \( \varepsilon \in (0,\beta_0^{(\text{conf})}-\beta).\) Then
    \begin{equation*}
        \begin{split}
            & \sum_{\mathcal{S} \in \Xi_{1+,k+,p}} \bigl| \Psi^\gamma_{\beta,\kappa}(\mathcal{S}) \bigr|   
            \leq  
            C_{\varepsilon,n}^{(2)}\Bigl(\frac{\varphi_{\beta}(1)}{\varphi_{\beta+\varepsilon}(1)}\Bigr)^k ,
        \end{split}
    \end{equation*} 
    where \( C_{\varepsilon,n}^{(2)}\) is defined in~\eqref{eq: cbeta* cp}.
\end{lemma}

\section{Proof of Theorem~\ref{theorem: general perimeter}}\label{section: general perimeter law}

In this section, we provide a proof of Theorem~\ref{theorem: general perimeter}.
The main tools in these proofs are the high-temperature expansion given by Proposition~\ref{proposition: high high mn}, together with and the following lemma.

\begin{lemma}\label{lemma: increasing}
	Let \( \beta,\kappa \geq 0, \) and let \( \gamma \in C^1(B_N,\mathbb{Z}) \) be a path. Then \( \mathbb{E}_{N,\beta,\kappa}[W_{\gamma}] \) is increasing in both \( \beta \) and \( \kappa. \) 
\end{lemma}

\begin{proof}

	By Griffith's second inequality, for all paths \( \gamma',\gamma'' \in C^1(B_N,\mathbb{Z}), \) we have
	\[
	\mathbb{E}_{N,\beta,\kappa}[ W_{\gamma'+\gamma''} ] \geq \mathbb{E}_{N,\beta,\kappa}[ W_{\gamma'} ] \mathbb{E}_{N,\beta,\kappa}[ W_{\gamma''} ].
	\]
	(For a reference, see, e.g., \cite[Proposition 6.3]{fv2025} for the case \( \kappa=0.\) The proof for \( \kappa>0 \) is analogous.)
	Since 
	\[
		\mathbb{E}_{N,\beta,\kappa}[W_\gamma] = \frac{\sum_{\sigma \in \Omega_1(B_N,G)} \rho(\sigma(\gamma)) e^{\beta \sum \rho(d\sigma(p)) + \kappa \sum \rho(\sigma(e))}}{\sum_{\sigma \in \Omega_1(B_N,G)} e^{\beta \sum \rho(d\sigma(p)) + \kappa \sum \rho(\sigma(e))}},
	\]
	it follows that 
	\[
		\frac{d}{d\beta} \mathbb{E}_{N,\beta,\kappa}[W_\gamma] = 
		\sum_{p} \mathbb{E}_{N,\beta,\kappa}[W_{\gamma+ \partial p}] 
		-
		\mathbb{E}_{N,\beta,\kappa}[W_\gamma]
		\sum_{p} \mathbb{E}_{N,\beta,\kappa}[W_{ \partial p}] \geq 0
	\]
	and
	\[
		\frac{d}{d\kappa} \mathbb{E}_{N,\beta,\kappa}[W_\gamma]
		=
		\sum_e \mathbb{E}_{N,\beta,\kappa}[W_{\gamma+e}]
		-
		\mathbb{E}_{N,\beta,\kappa}[W_{\gamma}]
		\sum_e \mathbb{E}_{N,\beta,\kappa}[W_{e}] \geq 0.
	\]
	This concludes the proof.
\end{proof}

We now give a proof of Theorem~\ref{theorem: general perimeter}.
\begin{proof}[Proof of Theorem~\ref{theorem: general perimeter}]
	We first show that~Theorem~\ref{theorem: general perimeter}\ref{item: general perimeter a} holds. To this end, note that for each \( \kappa >0, \) we have \( \bar \varphi_\kappa>0. \) Consequently, for each \( \omega' \in \Omega_2(B_N,\mathbb{Z}_{m\vee n}), \) we have
	\[
	\prod_{e \in \gamma}\frac{\bar \varphi_{\kappa,m\lor n}\bigl( \delta \omega'(e)-\gamma[e]\bigr)}{\bar \varphi_{\kappa,m\lor n} \bigl(\delta \omega' (e)\bigr)} \geq \biggl( \frac{\min_{j \in \mathbb{Z}_{m \vee n}}\bar \varphi_{\kappa,m\lor n}(j)}{\max_{j \in \mathbb{Z}_{m \vee n}}\bar \varphi_{\kappa,m\lor n}(j)} \biggr)^{|\gamma|}.
	\] 
	Using Proposition~\ref{proposition: high high mn}, it follows that there is  \( b_{\beta,\kappa} > 0 \) such that
	\[
	\langle W_{\gamma}  \rangle_{\beta,\kappa} \geq e^{-b_{\beta,\kappa}|\gamma|}.
	\]  
	%Let \( e \) an edge with one end-point in the origin. Then, by the translation invariance of \( \langle \cdot \rangle_{\beta,\kappa} \) and Griffith's inequality, we have  
	%\begin{equation*}
		%\langle W_{\gamma_n} \rangle_{\beta,\kappa} \geq \bigl\langle \rho(\sigma(e)) \bigr\rangle_{\beta,\kappa}^{|\gamma_n|} ,
	%\end{equation*}
	%letting \( b'_{\beta,\kappa} \coloneqq -\log \langle W_{\gamma_1} \rangle_{\beta,\kappa}, \) 
	For the other direction, we simply note that by definition, we have \( \langle W_{\gamma}\rangle_{\beta,\kappa} \leq 1. \) 
	This completes the proof of~Theorem~\ref{theorem: general perimeter}\ref{item: general perimeter a}.
	
	To see that Theorem~\ref{theorem: general perimeter}\ref{item: general perimeter b} holds, recall first that by Griffith's second inequality, for any paths \( \gamma ,\gamma' \in C^1(\mathbb{Z}^4,\mathbb{Z})\) with finite support, we have 
	\[
		\langle W_{\gamma+\gamma'} \rangle_{\beta,\kappa} \geq \langle W_{\gamma} \rangle_{\beta,\kappa} \langle W_{\gamma'} \rangle_{\beta,\kappa},
	\]
	and hence
	\[
	\frac{\log \langle W_{\gamma + \gamma'}\rangle_{\beta,\kappa}}{|\gamma|+|\gamma'|}  \geq \frac{|\gamma|}{|\gamma|+|\gamma'|} \frac{\log \langle W_\gamma \rangle_{\beta,\kappa}}{|\gamma|} + \frac{|\gamma'|}{|\gamma|+|\gamma'|} \frac{\log \langle W_{\gamma'} \rangle_{\beta,\kappa}}{|\gamma'|}.
	\]
	Now assume that for \( j \geq 1, \) \( \gamma_j \) is a straight line of length \(j.\) Then the above equation implies in particular that the sequence \( {\log \langle W_{\gamma_{j}}\rangle_{\beta,\kappa}}/{|\gamma_j|}, \) \( j \geq 1, \) is superadditive.
	Since, by~Theorem~\ref{theorem: general perimeter}\ref{item: general perimeter a}, \( \log \langle {W_{\gamma_j}} \rangle_{\beta,\kappa}/|\gamma_j|  \) is bounded from above and below, uniformly in \( n, \) it follows that \( \lim_{j \to \infty} \frac{\log \langle W_{\gamma_j} \rangle_{\beta,\kappa}}{|\gamma_j|} \in (0,\infty ) \) exists, which concludes the proof in this case.
	
	The proof for \( \gamma_j \) being a rectangular loop with side lengths \( aj \) and \( bj \) for some \( a,b \in \mathbb{N} \) follows analogously, and is hence omitted here.
	%
	%
	%
	%Next, assume that \( \gamma_{j,k} \) is a rectangular loop with side lengths \( aj \) nd \( bk \) for some \( a,b \in \mathbb{N}. \) Then, again by Griffith's second inequality, we have
	%\begin{align*}
	%	\frac{\log \langle W_{\gamma_{j_1+j_2,h}}\rangle}{2(j_1+j_2+k)} \geq \log \langle W_{\gamma_{j_1,k}}\rangle +\log \langle W_{\gamma_{j_2,k}}\rangle 
	%\end{align*}
	% 
\end{proof}

\section{Proofs of Theorem~\ref{theorem: Higgs phase} and Theorem~\ref{theorem: main confinement}}\label{section: proofs of the main results}

In this section, we will prove Theorem~\ref{theorem: Higgs phase} and Theorem~\ref{theorem: main confinement}. The main tool in these proofs will be the cluster expansions of Section~\ref{section: cluster expansions in Higgs phase} and Section~\ref{section: cluster expansions in confinement phase}; hence, we use the notation of those corresponding sections here.
We state and prove both theorems for straight lines only, as the proof in the case of a loop is completely analogous and follows by considering corners instead of endpoints of the path. We further mention that the same strategy can easily be adapted to give analogous results also for, e.g., sequences of paths of the same shape whose side lengths are proportional to \(n. \)

We now introduce three paths that will be useful in the proofs. We will view paths  as elements in \( C^1(\mathbb{Z}^d,\mathbb{Z}), \) often restricted to \( C^1(B_N,\mathbb{Z}) \) in the natural way.
\begin{itemize}
		\item Let \( \gamma_\infty \) be the bi-infinite line placed at the \( x \)-axis.  
		
		\item Let \( \gamma_\infty^+ \) be the restriction of \( \gamma_\infty \) to the positive \( x \)-axis. Without loss of generality we can assume that \( \gamma_n \) has one end-point at the origin and that its support is a subset of the support of \( \gamma_\infty^+. \) We let \( \gamma_\infty^- \) be the restriction of \( \gamma_\infty \) to \(  \support( \gamma_\infty - \gamma_\infty^+ + \gamma_n).\)

		\item Let \( e_0 \) be the unique edge in \( \gamma_\infty^+ \) that has one endpoint at the origin. For \( e \in \gamma_\infty, \) we let \( \tau_e \) be the translation that maps \( e_0 \) to~\( e .\) When \( c \in C^1(B_N,\mathbb{Z}) \) is such that \( \tau_e e' \in B_N \) for all \( e' \in \support c, \) we define \( \tau_e c \) by \( \tau_e c(e') = c(\tau_e e') . \)
	\end{itemize}

\subsection{Proof of Theorem~\ref{theorem: Higgs phase} (the Higgs phase)}

In this section, we give a proof of Theorem~\ref{theorem: Higgs phase}. Throughout this section, we use the notation of Section~\ref{section: cluster expansions in Higgs phase}.

\begin{proof}[Proof of Theorem~\ref{theorem: Higgs phase}] 
    By combining~\eqref{eq: log Z expansion},~\eqref{eq:dec5.1}, and~\eqref{eq: log ZUgamma},  we can write
	\begin{equation}\label{eq: expansion equation}
		\begin{split}
        	&-\log \mathbb{E}^{(U)}_{N,\beta,\kappa}[ W_{\gamma_n} ] = \sum_{\mathcal{S} \in \Xi^{(N)}} \bigl( \Psi_{\beta,\kappa}(\mathcal{S})-\Psi_{\beta,\kappa,\gamma_n}(\mathcal{S})\bigr) = 
        	\sum_{\mathcal{S} \in \Xi^{(N)}} \Psi_{\beta,\kappa}(\mathcal{S})\pigl( 1-\rho \bigl(\mathcal{S}(\gamma_n)\bigr) \pigr) 
        	\\&\qquad= 
        	\sum_{e \in \gamma}\sum_{\mathcal{S} \in \Xi^{(N)}_{1^+,1^+,e}} \frac{\Psi_{\beta,\kappa}(\mathcal{S})\pigl( 1-\rho \bigl(\mathcal{S}(\gamma_n)\bigr) \pigr)}{ |\support \mathcal{S} \cap \support \gamma_n|}.
        \end{split}
    \end{equation}   
    To simplify notation, for a path \( \gamma \in C^1(B_N,\mathbb{Z})\) and a cluster \( \mathcal{S}\in \Xi, \) we let
    \begin{equation*}
    	F_{\beta,\kappa}(\mathcal{S},\gamma) \coloneqq \frac{\Psi_{\beta,\kappa}(\mathcal{S})\pigl( 1-\rho \bigl(\mathcal{S}(\gamma)\bigr) \pigr)}{ |\support \mathcal{S} \cap \support \gamma|}.
    \end{equation*}
    Let \( M>|\gamma_n| \) be such that the distance between \( \support \gamma_n \) and the boundary of \( B_N \) is at least \( M. \) Then
    \begin{align*}
		& -\log \mathbb{E}^{(U)}_{N,\beta,\kappa}[ W_{\gamma_n} ] = \underbrace{\sum_{e \in \gamma_n}\sum_{\mathcal{S} \in \Xi^{(N)}_{1^+,M^+,e}} F_{\beta,\kappa}(\mathcal{S},\gamma_n)}_{ \eqqcolon E_1 }
		+
        \underbrace{
        \sum_{e \in \gamma_n}  \sum_{\mathcal{S} \in \Xi_{1^+,M^-,e}^{(N)}} F_{\beta,\kappa}(\mathcal{S},\gamma_n)  }_{\eqqcolon S_1}.
	\end{align*}
	For any \( e \in \gamma_n, \) we have \( \Xi_{1^+,M^-,e}^{(N)} = \Xi_{1^+,M^-,e}^{(\infty)}, \) 
	and hence
	\begin{align*}
		&S_1 = \sum_{e \in \gamma_n}  \sum_{\mathcal{S} \in \Xi_{1^+,M^-,e}^{(\infty)}} F_{\beta,\kappa}(\mathcal{S},\gamma_n) 
		\\&\qquad=
		\underbrace{\sum_{e \in \gamma_n}  \sum_{\mathcal{S} \in \Xi_{1^+,M^-,e}^{(\infty)}} F_{\beta,\kappa}(\mathcal{S},\gamma_\infty) }_{\eqqcolon S_2}
		+
		\underbrace{\sum_{e \in \gamma_n}  \sum_{\mathcal{S} \in \Xi_{1^+,M^-,e}^{(\infty)}} \Bigl( F_{\beta,\kappa}(\mathcal{S},\gamma_n)  - F_{\beta,\kappa}(\mathcal{S},\gamma_\infty)  \Bigr)}_{\eqqcolon S_3}.
	\end{align*}
	By symmetry, we further have 
	\begin{align*}
		&S_2
		=
		|\gamma_n|\sum_{\substack{\mathcal{S} \in \Xi_{1^+,M^-,e_0}^{(\infty)} }} F_{\beta,\kappa}(\mathcal{S},\gamma_\infty) 
		\\&\qquad=
		|\gamma_n|\sum_{\substack{\mathcal{S} \in \Xi_{1^+,1^+,e_0}^{(\infty)} }} F_{\beta,\kappa}(\mathcal{S},\gamma_\infty) 
		-
		\underbrace{|\gamma_n|\sum_{\substack{\mathcal{S} \in \Xi_{1^+,M^+,e_0}^{(\infty)} }} F_{\beta,\kappa}(\mathcal{S},\gamma_\infty) }_{\eqqcolon E_2}.
	\end{align*} 
	Let
	\begin{equation*}
		a_{\beta,\kappa} \coloneqq \sum_{\mathcal{S} \in \Xi_{1^+,1^+,e_0}^{(\infty)}} F_{\beta,\kappa}(\mathcal{S},\gamma_\infty). 
	\end{equation*}
	By Lemma~\ref{lemma: new upper bound}, \( a_{\beta,\kappa} \) is well defined and finite. Combining the above equations, we obtain
	\begin{align*}
		& -\log \mathbb{E}^{(U)}_{N,\beta,\kappa}[ W_{\gamma_n} ] -a_{\beta,\kappa}|\gamma_n|= E_1-E_2+S_3.
	\end{align*}
	
	We now return to the sum \( S_3. \) 
	To this end, note that if \( \mathcal{S} \in \Xi_{1^+,M^-,e}^{(\infty)}\) and \( \|\mathcal{S}\|_1 < \dist(e,\partial \gamma_n) , \) then \(  \mathcal{S}(\gamma_n) = \mathcal{S}(\gamma_\infty) \) and \( |\support   \mathcal{S} \cap \support \gamma_n| = |\support  \mathcal{S} \cap \support \gamma_\infty|,\) and hence \( F_{\beta,\kappa}(\mathcal{S},\gamma_n) = F_{\beta,\kappa}(\mathcal{S},\gamma_\infty).\) Using these observations, it follows that  
	\begin{align*}
		&S_3 =
		\sum_{e \in \gamma_n} \sum_{\substack{\mathcal{S} \in \Xi^{(\infty)}_{1^+,M^-,e} \mathrlap{\colon}\\  \dist(e,\partial \gamma_n) \leq  \| \mathcal{S}\|_1 }} 
		\Bigl( F_{\beta,\kappa}(\mathcal{S},\gamma_n) - F_{\beta,\kappa}(\mathcal{S},\gamma_\infty)\Bigr).
	\end{align*} 
	Next, we note that since \( |\gamma_n| \geq \dist (e,\partial \gamma_n) \) for all \( e \in \gamma_n, \) we have
	\begin{align*}
		&S_3 = \sum_{e \in \gamma_n} \sum_{\substack{\mathcal{S} \in \Xi^{(\infty)}_{1^+,M^-,e} \mathrlap{\colon}\\ \dist(e,\partial \gamma_n) \leq \| \mathcal{S}\|_1 < |\gamma_n|}}    \pigl( F_{\beta,\kappa}(\mathcal{S},\gamma_n) - F_{\beta,\kappa}(\mathcal{S},\gamma_\infty) \pigr) \qquad (\eqqcolon S_4)
		\\&\qquad\qquad+
		\sum_{e \in \gamma_n} \sum_{\substack{\mathcal{S} \in  \Xi^{(\infty)}_{1^+,M^-,e} \mathrlap{\colon}\\ \| \mathcal{S}\|_1 \geq |\gamma_n|  }}  \pigl( F_{\beta,\kappa}(\mathcal{S},\gamma_n) -F_{\beta,\kappa}(\mathcal{S},\gamma_\infty) \pigr) \qquad (\eqqcolon E_3)
	\end{align*}  
	Now assume that \( \mathcal{S} \in \Xi_{1^+,M^-,e}^{(\infty)},\) is such that \( \dist(e,\partial \gamma_n) \leq  \|\mathcal{S}\|_1 < |\gamma_n|. \) Then we have either  \(  \mathcal{S}(\gamma_n) = \mathcal{S}(\gamma_\infty^+) \) and \( |\support   \mathcal{S} \cap \support \gamma_n| = |\support  \mathcal{S} \cap \support \gamma_\infty^+| \)  or  \(  \mathcal{S}(\gamma_n) = \mathcal{S}(\gamma_\infty^-) \) and \( |\support   \mathcal{S} \cap \support \gamma_\infty^-  | = |\support  \mathcal{S} \cap \support \gamma_\infty^-|. \) Hence either \( F_{\beta,\kappa}(\mathcal{S},\gamma_n) = F_{\beta,\kappa}(\mathcal{S},\gamma_\infty^+)\) or \( F_{\beta,\kappa}(\mathcal{S},\gamma_n)=F_{\beta,\kappa}(\mathcal{S},\gamma_\infty^-) \) holds. 
	By symmetry, it follows that
	\begin{align*}
		&S_4
		=
		2\sum_{e \in \gamma_\infty^+} \sum_{\substack{\mathcal{S} \in \Xi^{(\infty)}_{1^+,M^-,e} \mathrlap{\colon}\\ \dist(e,\partial \gamma_\infty^+) \leq \| \mathcal{S}\|_1<|\gamma_n|}} \pigl( F_{\beta,\kappa}(\mathcal{S},\gamma_\infty^+) -F_{\beta,\kappa}(\mathcal{S},\gamma_\infty) \pigr)
		\\&\qquad=
		2\sum_{e \in \gamma_\infty^+} \sum_{\substack{\mathcal{S} \in \Xi^{(\infty)}_{1^+,1^+,e} \mathrlap{\colon}\\ \dist(e,\partial \gamma_\infty^+) \leq \| \mathcal{S}\|_1<|\gamma_n|}} \pigl( F_{\beta,\kappa}(\mathcal{S},\gamma_\infty^+) -F_{\beta,\kappa}(\mathcal{S},\gamma_\infty) \pigr).
	\end{align*}
	Let
	\begin{align*}
		C_{\beta,\kappa} \coloneqq 2\sum_{e \in \gamma_\infty^+} \sum_{\substack{\mathcal{S} \in \Xi^{(\infty)}_{1^+,1^+,e} \mathrlap{\colon}\\    \dist(e,\partial \gamma_\infty^+) \leq \| \mathcal{S}\|_1 }}  \pigl( F_{\beta,\kappa}(\mathcal{S},\gamma_\infty^+) -F_{\beta,\kappa}(\mathcal{S},\gamma_\infty) \pigr).
	\end{align*}
	By~Lemma~\ref{lemma: new upper bound}, \( C_{\beta,\kappa} \) is well defined and finite. Moreover, we have
	\begin{align*}
		&S_4 = C_{\beta,\kappa} 
		-
		\underbrace{2\sum_{e \in \gamma_\infty^+} \sum_{\substack{\mathcal{S} \in \Xi^{(\infty)}_{1^+,1^+,e} \mathrlap{\colon}\\  \| \mathcal{S}\|_1 \geq \max(|\gamma_n|,\dist(e,\partial \gamma_\infty^+))  }}   \pigl( F_{\beta,\kappa}(\mathcal{S},\gamma_\infty^+) -F_{\beta,\kappa}(\mathcal{S},\gamma_\infty) \pigr) }_{\eqqcolon E_4}.
	\end{align*}
	Combining the above equations, we obtain the upper bound
	\begin{align*}
		&\pigl| -\log \mathbb{E}^{(U)}_{N,\beta,\kappa}[ W_{\gamma_n} ] - a_{\beta,\kappa}|\gamma_n| -C_{\beta,\kappa}\pigr| \leq 
		|E_1| + |E_2| + |E_3| + |E_4|
		\\&\qquad\leq
		\sum_{e \in \gamma_n}\sum_{\mathcal{S} \in \Xi^{(N)}_{1^+,M^+,e}} \bigl| \Psi_{\beta,\kappa}(\mathcal{S})  \bigr|
		+|\gamma_n|\sum_{\substack{\mathcal{S} \in \Xi^{(\infty)}_{1^+,M^+,e_0} }} \bigl|\Psi_{\beta,\kappa}(\mathcal{S})   \bigr|
		\\&\qquad\qquad+  \sum_{e \in \gamma_n} \sum_{\substack{\mathcal{S} \in \Xi^{(\infty)}_{1^+,M^-,e} \mathrlap{\colon}\\  \| \mathcal{S}\|_1 \geq |\gamma_n|}} \bigl|\Psi_{\beta,\kappa}(\mathcal{S})  \bigr|
		+2
		\sum_{e \in \gamma_\infty^+} \sum_{\substack{\mathcal{S} \in \Xi^{(\infty)}_{1^+,1^+,e} \mathrlap{\colon}\\ \| \mathcal{S}\|_1 \geq \max(|\gamma_n|,\dist(e,\partial \gamma_\infty^+) )}} \bigl| \Psi_{\beta,\kappa}(\mathcal{S})     
		\bigr|.
	\end{align*}
	By Lemma~\ref{lemma: new upper bound}, under the assumptions of the theorem, all of these sums are finite. Letting first \( N \) and then \( M \) go to infinity, it follows that
	\begin{align*}
		&\pigl| -\log \langle W_{\gamma_n} \rangle_{\beta,\kappa} - a_{\beta,\kappa}|\gamma_n| -C_{\beta,\kappa}\pigr|  
		\\&\qquad\leq  
		\sum_{e \in \gamma_n} \sum_{\substack{\mathcal{S} \in \Xi^{(\infty)}_{1^+,1^+,e} \mathrlap{\colon}\\  \| \mathcal{S}\|_1 \geq |\gamma_n|}} \bigl|\Psi_{\beta,\kappa}(\mathcal{S})  \bigr|
		+2
		\sum_{e \in \gamma_\infty^+} \sum_{\substack{\mathcal{S} \in \Xi^{(\infty)}_{1^+,1^+,e} \mathrlap{\colon}\\ \| \mathcal{S}\|_1 \geq \max(|\gamma_n|,\dist(e,\partial \gamma_\infty^+) )}} \bigl| \Psi_{\beta,\kappa}(\mathcal{S})      
		\\&\qquad\leq 
		|\gamma_n|\sum_{\substack{\mathcal{S} \in \Xi^{(\infty)}_{1^+,|\gamma_n|^+,e_0} }} \bigl|\Psi_{\beta,\kappa}(\mathcal{S})  \bigr|
		+2
		\sum_{m=1}^\infty  \sum_{\substack{\mathcal{S} \in \Xi^{(\infty)}_{1^+,1^+,e_0} \mathrlap{\colon}\\ \| \mathcal{S}\|_1 \geq \max(|\gamma_n|,m) }} \bigl| \Psi_{\beta,\kappa}(\mathcal{S}) \bigr|    
		\bigr|.
	\end{align*}
	Using Lemma~\ref{lemma: new upper bound}, we can bound the right-hand side of the previous inequality from above by
	\begin{equation*}
		\frac{C_{\varepsilon} |\gamma_n|e^{-4|\gamma_n|(\kappa-\kappa_0^{(\textrm{Higgs})}-\varepsilon) }}{4}
		+
		\frac{C_{\varepsilon}|\gamma_n| e^{-4|\gamma_n|(\kappa-\kappa_0^{(\textrm{Higgs})}-\varepsilon) }}{2}
		+
		   \frac{C_{\varepsilon} e^{-4|\gamma_n| (\kappa-\kappa_0^{(\textrm{Higgs})}-\varepsilon) }}{2(1-e^{-4 (\kappa-\kappa_0^{(\textrm{Higgs})}-\varepsilon) })}.
	\end{equation*} 
	This concludes the proof.
\end{proof}

\subsection{Proof of Theorem~\ref{theorem: main confinement} (the confinement phase)}

In this section, we prove Theorem~\ref{theorem: main confinement}. Throughout this section, we use the notation of Section~\ref{section: cluster expansions in confinement phase}. 
The proof is very similar to the proof of Theorem~\ref{theorem: Higgs phase}, except the polymers are now different, causing a slightly different interaction between polymers and the path \( \gamma_n \) and also different upper bounds.
In the proof, we let
	\begin{equation*}
    	\Xi_{1,j,e}^{(N)}   \coloneqq \bigl\{\{ \eta \} \in \Xi_1^{(N)}\colon   e \in \support \delta\eta \text{ and } |(\support \eta)^+|  = j\bigr\},
	\end{equation*}  
	and define \( \Xi^{(N)}_{1^+,j^+,e}  \) etc.{} as before. 

\begin{proof}[Proof of Theorem~\ref{theorem: main confinement}]
	Recall the definitions of~\( \Psi^\gamma_{\beta,\kappa} \) and~\( \hat Z_{N,\beta,\kappa}[\gamma] \) from~\eqref{eq: Psi confinement} and~\eqref{eq: hat Z confinement} respectively. 
	By Proposition~\ref{proposition: high high mn}, for any \( n \geq 1 \), we have
	\begin{equation*}
		\frac{ Z_{N,\beta,\kappa}[\gamma_n]}{ Z_{N,\beta,\kappa}[0]}
		=
		\frac{\hat Z_{N,\beta,\kappa}[\gamma_n]}{\hat Z_{N,\beta,\kappa}[0]}.
	\end{equation*}  
	Using Lemma~\ref{lemma: upper bound and convergence}, it thus follows that
	\begin{align*}
		&-\log \mathbb{E}_{N,\beta,\kappa}[W_{\gamma_n}]  
		= 
		\log \hat Z_{N,\beta,\kappa}[\gamma_n] - \log \hat Z_{N,\beta,\kappa}[0] 
		\\&\qquad=\sum_{\mathcal{S} \in \Xi^{(N)}} \Psi^{\gamma_n}_{\beta,\kappa}(\mathcal{S}) - \sum_{\mathcal{S} \in \Xi^{(N)}} \Psi^0_{\beta,\kappa}(\mathcal{S})
		= 
		\sum_{\mathcal{S} \in \Xi^{(N)}_{1^+,1^+,e} }  \frac{\Psi^{\gamma_n}_{\beta,\kappa}(\mathcal{S}) - \Psi^0_{\beta,\kappa}(\mathcal{S}) }{ |\support \gamma_n \cap  \support \delta \mathcal{S}|} .
	\end{align*}
	To simplify notation, for a path \( \gamma \) and \( \mathcal{S} \in \Xi, \) we let 
	\begin{equation*}
		G_{\beta,\kappa}(\mathcal{S},\gamma) \coloneqq \frac{\Psi^{\gamma}_{\beta,\kappa}(\mathcal{S}) - \Psi^0_{\beta,\kappa}(\mathcal{S}) }{ |\support \gamma \cap  \support \delta \mathcal{S}|}.
	\end{equation*}
	Since
 	\begin{equation*}
 		\Psi^{\gamma}_{\beta,\kappa}(\mathcal{S})  = \varphi_\kappa(1)^{-2|\support \gamma \cap \support \delta \mathcal{S}|}\Psi^0_{\beta,\kappa}(\mathcal{S}),
 	\end{equation*}
 	the function \( G_{\beta,\kappa}(\mathcal{S},\gamma) \) depends on \( \gamma \) only through \( |\support \gamma \cap \support \delta \mathcal{S}|. \)

	Let \( M > |\gamma_n|\) be such that the distance between \( \support \gamma_n \) and \( \partial B_N \) is at least \( M.\) Then 
	\begin{align*}
		&-\log \mathbb{E}_{N,\beta,\kappa}[W_{\gamma_n}] - |\gamma_n| \log \varphi_\kappa\bigl( 1\bigr) 
		\\&\qquad
		=
		\underbrace{\sum_{e \in \gamma_n} \sum_{\mathcal{S} \in \Xi_{1^+,M^+,e}^{(N)} } G_{\beta,\kappa}(\mathcal{S},\gamma_n)}_{ \eqqcolon E_1}
		+
		\underbrace{\sum_{e \in \gamma_n} \sum_{\substack{\mathcal{S} \in \Xi_{1^+,M^-,e}^{(N)}  
		}} G_{\beta,\kappa}(\mathcal{S},\gamma_n)}_{\eqqcolon S_1}.
	\end{align*}
	Next, note that for any \( e \in \gamma_n, \) we have \( \mathcal{S} \in \Xi_{1^+,M^-,e}^{(N)}=\mathcal{S} \in \Xi_{1^+,M^-,e}^{(\infty)} , \) and hence
	\begin{align*}
	&S_1 =    
	\sum_{e \in \gamma_n} \sum_{\mathcal{S} \in \Xi_{1^+,M^-,e}^{(\infty)}  } G_{\beta,\kappa}(\mathcal{S},\gamma_n) 
	\\&\qquad=
		\underbrace{\sum_{e \in \gamma_n} \sum_{\substack{\mathcal{S} \in \Xi_{1^+,M^-,e}^{(\infty)} }}   G_{\beta,\kappa}(\mathcal{S},\gamma_\infty)}_{\eqqcolon S_2}
		+ 
		\underbrace{\sum_{e \in \gamma_n} \sum_{\substack{\mathcal{S} \in \Xi^{(\infty)}_{1^+,M^-,e}  }} \!\!\!\!  \pigl( G_{\beta,\kappa}(\mathcal{S},\gamma_n)- G_{\beta,\kappa}(\mathcal{S},\gamma_\infty) \pigr)}_{\eqqcolon S_3}. 
	\end{align*}
	Note that by symmetry.
	\begin{align*}
		&S_2 = |\gamma_n|\sum_{\mathcal{S} \in \Xi^{(\infty)}_{1^+,M^-,e_0} } G_{\beta,\kappa}(\mathcal{S},\gamma_\infty)
		\\&\qquad=
		|\gamma_n|\sum_{\mathcal{S} \in \Xi^{(\infty)}_{1^+,1^+,e_0} } G_{\beta,\kappa}(\mathcal{S},\gamma_\infty)
		-
		\underbrace{|\gamma_n|\sum_{\mathcal{S} \in \Xi^{(\infty)}_{1^+,M^+,e_0} } G_{\beta,\kappa}(\mathcal{S},\gamma_\infty) }_{\eqqcolon E_2}.
	\end{align*}
	With this in mind, define
	\[
		a_{\beta,\kappa} \coloneqq \log \varphi_\kappa(1) +  \sum_{\mathcal{S} \in \Xi^{(\infty)}_{1^+,1^+,e_0} } G_{\beta,\kappa}(\mathcal{S},\gamma_\infty).
	\]
	By Lemma~\ref{lemma: general Cbeta cp}, \( a_{\beta,\kappa} \) is well defined and finite.
	Combining the previous equations, we have
 	\begin{align*} 
		-\log \mathbb{E}_{N,\beta,\kappa}[W_{\gamma_n}] - a_{\beta,\kappa} |\gamma|
		=
		E_1-E_2+S_3.
 	\end{align*}  
 	We now return to the sum \( S_3. \) 
	To this end, note that if \( \mathcal{S} \in \Xi_{1^+,M^-,e}^{(\infty)}\) and \( \|\mathcal{S}\|_1 < \dist(e,\partial \gamma_n) , \) then \(  |\support \gamma_n \cap \support \delta \mathcal{S}|=|\support \gamma_\infty \cap \support \delta \mathcal{S}| \),  and hence \( G_{\beta,\kappa}(\mathcal{S},\gamma_n) = G_{\beta,\kappa}(\mathcal{S},\gamma_\infty).\) Using these observations, it follows that 
	\begin{align*}
		&S_3 =
		\sum_{e \in \gamma_n} \sum_{\substack{\mathcal{S} \in \Xi^{(\infty)}_{1^+,M^-,e} \mathrlap{\colon}\\ \| \mathcal{S}\|_1 \geq \dist(e,\partial \gamma_n) }} 
		\Bigl( G_{\beta,\kappa}(\mathcal{S},\gamma_n) - G_{\beta,\kappa}(\mathcal{S},\gamma_\infty)\Bigr).
	\end{align*} 
	Next, we note that
	\begin{align*}
		&S_3 = \sum_{e \in \gamma_n} \sum_{\substack{\mathcal{S} \in \Xi^{(\infty)}_{1^+,M^-,e} \mathrlap{\colon}\\ |\gamma_n| >\| \mathcal{S}\|_1 \geq \dist(e,\partial \gamma_n)  }}    \pigl( G_{\beta,\kappa}(\mathcal{S},\gamma_n) - G_{\beta,\kappa}(\mathcal{S},\gamma_\infty) \pigr) \qquad (\eqqcolon S_4)
		\\&\qquad\qquad+
		\sum_{e \in \gamma_n} \sum_{\substack{\mathcal{S} \in \Xi^{(\infty)}_{1^+,M^-,e}  \mathrlap{\colon}\\ \| \mathcal{S}\|_1 \geq \max(|\gamma_n|,\dist(e,\partial \gamma_n)) }}  \pigl( G_{\beta,\kappa}(\mathcal{S},\gamma_n) -G_{\beta,\kappa}(\mathcal{S},\gamma_\infty) \pigr). \qquad (\eqqcolon E_3)
	\end{align*} 
	Next, observe that if \( \mathcal{S} \in \Xi_{1^+,M^-,e}^{(\infty)},\)  \( \dist(e,\partial \gamma_n) <  \|\mathcal{S}\|_1 < |\gamma_n|, \) then we have either  \(  |\support \gamma_n \cap \support \delta \mathcal{S}|=|\support \gamma_\infty^+ \cap \support \delta \mathcal{S}| \) or  \(  |\support \gamma_n \cap \support \delta \mathcal{S}|=|\support \gamma_\infty^- \cap \support \delta \mathcal{S}| \) and hence either \( G_{\beta,\kappa}(\mathcal{S},\gamma_n) = G_{\beta,\kappa}(\mathcal{S},\gamma_\infty^+)\) or \( G_{\beta,\kappa}(\mathcal{S},\gamma_n)=G_{\beta,\kappa}(\mathcal{S},\gamma_\infty^-). \) 
	By symmetry, it follows that
	\begin{align*}
		&S_4
		=
		2\sum_{e \in \gamma_\infty^+} \sum_{\substack{\mathcal{S} \in \Xi^{(\infty)}_{1^+,M^-,e} \mathrlap{\colon}\\  \dist(e,\partial \gamma_\infty^+) \leq \| \mathcal{S}\|_1 < |\gamma_n|} } \pigl( G_{\beta,\kappa}(\mathcal{S},\gamma_\infty^+) -G_{\beta,\kappa}(\mathcal{S},\gamma_\infty) \pigr).
	\end{align*}
	Let
	\begin{align*}
		C_{\beta,\kappa} \coloneqq 2\sum_{e \in \gamma_\infty^+} \sum_{\substack{\mathcal{S} \in \Xi^{(\infty)}_{1^+,1^+,e} \mathrlap{\colon}\\ \| \mathcal{S}\|_1  \geq \dist(e,\partial \gamma_\infty^+)  }}  \pigl( G_{\beta,\kappa}(\mathcal{S},\gamma_\infty^+) -G_{\beta,\kappa}(\mathcal{S},\gamma_\infty) \pigr).
	\end{align*}
	By~Lemma~\ref{lemma: new upper bound cp}, \( C_{\beta,\kappa} \) is well defined and finite. Moreover, we have
	\begin{align*}
		&S_4 = C_{\beta,\kappa} 
		-
		\underbrace{2\sum_{e \in \gamma_\infty^+} \sum_{\substack{\mathcal{S} \in \Xi^{(\infty)}_{1^+,M^-,e} \mathrlap{\colon}\\ \| \mathcal{S}\|_1 \geq \max(\dist(e,\partial \gamma_\infty^+) ,|\gamma_n|) }}   \pigl( G_{\beta,\kappa}(\mathcal{S},\gamma_\infty^+) -G_{\beta,\kappa}(\mathcal{S},\gamma_\infty) \pigr) }_{\eqqcolon E_4}.
	\end{align*}
	Combining the above equations, we have
 	\begin{align*} 
		-\log \mathbb{E}_{N,\beta,\kappa}[W_{\gamma_n}] - a_{\beta,\kappa} |\gamma_n| - C_{\beta,\kappa}
		=
		E_1-E_2+E_3-E_4+E_5.
 	\end{align*} 
 	Letting first \( N \) and then \( M \) tend to go infinity, by noting that \( E_1 \) and \( E_2 \) and \( E_4 \) tend to zero as \( M \to \infty, \) we thus obtain
 	\begin{align*} 
		&\bigl| -\log \langle W_{\gamma_n}\rangle_{\beta,\kappa} - a_{\beta,\kappa} |\gamma_n| - C_{\beta,\kappa} \bigr|
		\\&\qquad\leq 
		\sum_{e \in \gamma_n} \sum_{\substack{\mathcal{S} \in \Xi^{(\infty)}_{1^+,1^+,e} \\ \| \mathcal{S}\|_1\geq |\gamma_n| }} \!\!\!\!    \pigl( G_{\beta,\kappa}(\mathcal{S},\gamma_n)- G_{\beta,\kappa}(\mathcal{S},\gamma_\infty) \pigr) \qquad (\eqqcolon E_6)
		\\&\qquad\qquad+
		2\sum_{e \in \gamma_\infty^+} \sum_{\substack{\mathcal{S} \in \Xi^{(\infty)}_{1^+,1^+,e} \\ \| \mathcal{S}\|_1 \geq |\gamma_n| }} \!\!\!\!   \pigl( G_{\beta,\kappa}(\mathcal{S},\gamma_\infty^+) -G_{\beta,\kappa}(\mathcal{S},\gamma_\infty) \pigr) \qquad  (\eqqcolon E_7) .
 	\end{align*} 
 	 	To complete the proof, we now upper bound the right-hand side of the previous equation. This this end, we note that since \( \varphi_\kappa(1) < 1 , \) we have
	\begin{align*}
		|E_6| \leq |\gamma_n | \sum_{\substack{\mathcal{S} \in \Xi^{(\infty)}_{1^+,1^+,e_0} \\ \| \mathcal{S}\|_1\geq |\gamma_n| }} \!\!\!\! \bigl|\Psi_{\beta,\kappa}^{\gamma_n} (\mathcal{S}) \bigr|  
		+
		|\gamma_n | \sum_{\substack{\mathcal{S} \in \Xi^{(\infty)}_{1^+,1^+,e_0} \\ \| \mathcal{S}\|_1\geq |\gamma_n| }} \!\!\!\!  \bigl|\Psi_{\beta,\kappa}^{\gamma_\infty} (\mathcal{S}) \bigr| .
	\end{align*}
	Noting that
	\[
	\Xi_{1^+,k^+,e_0} \subseteq \bigcup_{p \in \partial e_0} \Xi_{1^+,k^+,p}
	\]
	and using that \( |\partial e_0| = 2(d-1), \)
	it follows from Lemma~\ref{lemma: new upper bound cp} that
	\begin{align*}
		|E_6| \leq 4  
		C_{\varepsilon,n}^{(2)}(d-1) |\gamma_n |\Bigl(\frac{\varphi_{\beta}(1)}{\varphi_{\beta+\varepsilon}(1)}\Bigr)^{|\gamma_n|} 
	\end{align*}
	which tends to zero exponentially fast in \( |\gamma_n|. \)
	
	Next, we note that if \( e \in \gamma_\infty^+ \) and \( \mathcal{S} \in \Xi^{(\infty)}_{1^+,1^+,e} \) is such that \( |\support \delta\mathcal{S} \cap \support \gamma_\infty| \neq |\support \delta\mathcal{S} \cap \support \gamma_\infty^+|, \) then we must have \( \|\mathcal{S}\|_1 \geq \dist(e,\partial \gamma_\infty^+). \) Noting that for each \( j \in \{ 0,1, \dots \} , \) there is a unique edge \( e \in \gamma_\infty^+ \) with \( \dist(e,\partial \gamma_\infty^+) = j, \) we obtain
	\begin{align*}
		&|E_7| \leq 2  \sum_{e \in \gamma_\infty^+ } \sum_{\substack{\mathcal{S} \in \Xi^{(\infty)}_{1^+,1^+,e} \\ \| \mathcal{S}\|_1 \geq \max(|\gamma_n|,\dist(e,\partial \gamma_\infty^+)) }} \!\!\!\! \Bigl(\bigl|\Psi_{\beta,\kappa}^{\gamma_\infty^+}(\mathcal{S}) \bigr| + \bigl|\Psi_{\beta,\kappa}^{\gamma_\infty}(\mathcal{S})  \bigr| \Bigr)
		\\&\qquad=
		2\sum_{j=0}^\infty
		\sum_{\substack{\mathcal{S} \in \Xi^{(\infty)}_{1^+,1^+,e_0} \\ \| \mathcal{S}\|_1 \geq \max(|\gamma_n|,j) }} \!\!\!\! \Bigl( \bigl|\Psi_{\beta,\kappa}^{\gamma_\infty^+}(\mathcal{S}) \bigr| + \bigl|\Psi_{\beta,\kappa}^{\gamma_\infty}(\mathcal{S})  \bigr| \Bigr).
	\end{align*}
	Again using Lemma~\ref{lemma: new upper bound cp}, it follows that
	\begin{align*}
		&|E_7| \leq 
		16(d-1)\sum_{j=0}^\infty
		\Bigl(\frac{\varphi_{\beta}(1)}{\varphi_{\beta+\varepsilon}(1)}\Bigr)^{\max(|\gamma_n|,j)} 
		\\&\qquad\leq 
		16(d-1)C_{\varepsilon,n}^{(2)}|\gamma_n|
		\Bigl(\frac{\varphi_{\beta}(1)}{\varphi_{\beta+\varepsilon}(1)}\Bigr)^{|\gamma_n|} 
		+
		\frac{16(d-1)C_{\varepsilon,n}^{(2)}\Bigl(\frac{\varphi_{\beta}(1)}{\varphi_{\beta+\varepsilon}(1)}\Bigr)^{|\gamma_n|}}{1-\Bigl(\frac{\varphi_{\beta}(1)}{\varphi_{\beta+\varepsilon}(1)}\Bigr)}.
	\end{align*}
	This completes the proof.
\end{proof}

\section*{Funding and conflicts of interest}

The author acknowledges support from the Swedish Research Council, grant number 2024-04744. There is no conflict of interest.

\section*{Data availability statement}
The paper has no associated data.

\cleardoublepage

\appendix
\section{Extensions of Theorems~\ref{theorem: general perimeter}, \ref{theorem: Higgs phase}, and~\ref{theorem: main confinement} to general finite abelian groups}\label{appendix: general theory}

 In this appendix, we outline how the proofs of Theorems~\ref{theorem: general perimeter}, \ref{theorem: Higgs phase}, and~\ref{theorem: main confinement} can be extended to finite abelian groups.  To this end, we first note that the starting point for the proofs of both Theorem~\ref{theorem: general perimeter} and Theorem~\ref{theorem: main confinement} was the high temperature expansions described in Sections~\ref{sec: high temperature ZmZn} and~\ref{sec: high temperature Z2}. The main tool for these expansions is the Fourier expansions for \( \mathbb{Z}_n \), which also exist for finite abelian groups. With these expansions at hand, all other proofs in the paper extend immediately to the finite abelian case, with the caveat that, in particular, constants and lower bounds are less concrete. The crucial property we use for most proofs is that the edges that are assigned a spin \( \sigma_e \) with \(  \tr \rho(\sigma_e)) \neq 1 \) will form clusters whose cardinality is very unlikely to be large.

 For completeness, we now describe the high temperature expansion for a general finite abelian group \( G\) in the case \( G=H , \) and mention how it extends the proof of Theorem~\ref{theorem: general perimeter} to this more general case.
 To this end, assume that \( G  \) is a finite and abelian group.  Let $\widehat G = \mathrm{Hom}(G,U(1))$ denote the corresponding  dual group of characters.
  Since $G$ is finite and abelian, the characters form an orthonormal basis of functions on $G$. Hence every function $f \colon G\to\mathbb C$ admits a Fourier expansion 
	\[
		f(g)
		=
		\sum_{\chi\in\widehat G} \widehat f_a(\chi)\chi(g),
	\]
	where, 
	 	\[
		\widehat f(\chi)
		=
		\frac{1}{|G|}
		\sum_{g\in G}
		f(g) \overline{\chi(g)}.
		\]
 	We note that~\eqref{eq: first Fourier expansion} is exactly such an expansion in the special case \( G = \mathbb{Z}^n. \)

Using this Fourier expansion, the following result is the general version of Proposition~\ref{proposition: high high mn} in the special case \( G=H. \)

\begin{proposition}\label{proposition: hte for general groups}
	Let \( G=H\) be a finite abelian group, and let \( \rho_G = \rho_H = \rho \) be a unitary representation of \( G. \) Further, let $\widehat G = \mathrm{Hom}(G,U(1))$ denote the corresponding dual group of characters. For \( g \in G \) and \( a >0, \) let \( f_a \colon G \to \mathbb{C}\) be the function \( g \mapsto e^{2a \real\tr \rho(g)} . \)
	Then
		\begin{align}
	Z^{(U)}_{\beta,\kappa,N}[\gamma] 
	&=
		|G|^{|C_1(B_N)^+|} \!\!\!\!\!\!\!\!
		\sum_{\substack{(\chi_p) \in {\widehat G}^{C_2(B_N)^+}}}
		\prod_{e \in \support \gamma} \frac{\widehat f_\kappa(\chi_{e,(\chi_p),\gamma} )  }{\widehat f_\kappa(\chi_{e,(\chi_p),0} )  }
		\\&\qquad\qquad\qquad\qquad\qquad\qquad \cdot \prod_{e \in C_1(B_N)^+} \widehat f_\kappa(\chi_{e,(\chi_p),0}  )  \prod_{p \in C_2(B_N)^+} \widehat f_\beta(\chi_p) ,\label{eq: gen fourier line 2}
	\end{align}
	where 
	\begin{align*}
		\chi_{e,(\chi_p),\gamma} (g) \coloneqq \overline{(\tr \rho(g))^{\gamma[e]} \prod_{p \in \support \hat \partial e}\chi_{p}(\partial p[e] g)   \Bigr)}.
	\end{align*} 
\end{proposition}

Note that since the function \( f_a \) in Proposition~\ref{proposition: hte for general groups} is strictly positive definite for every \( a > 0, \) all its Fourier coefficients are strictly non-negative. This allows us to think of the terms in~\eqref{eq: gen fourier line 2} as the weights of a probability measure on \( (\hat G)^{C_2(B_N)^+}.\) Using this observation, it immediately follows that
\begin{equation*}
	\mathbb{E}_{\beta,\kappa,N}[W_\gamma] \geq \inf_{(\chi_p) \in \hat G^{C_2(B_N)^+}} \prod_{e \in \support \gamma} \frac{\widehat f_\kappa(\chi_{e,(\chi_p),\gamma} )  }{\widehat f_\kappa(\chi_{e,(\chi_p),0} )  } \geq 
	\biggl( \inf_{(\chi_p) \in \hat G^{C_2(B_N)^+}}  \frac{\widehat f_\kappa(\chi_{e,(\chi_p),1 \cdot e} )  }{\widehat f_\kappa(\chi_{e,(\chi_p),0} )  }  \biggr)^{|\gamma|}
	,
\end{equation*}
which analogous to the lower bound in Theorem~\ref{theorem: general perimeter}. The rest of the proof of Theorem~\ref{theorem: general perimeter} in this more general setting is completely analogous to the proof for the case \( G = \mathbb{Z}_m \) and \( H = \mathbb{Z}_n, \) and is thus omitted here.

We now proceed to the proof of Proposition~\ref{proposition: hte for general groups}.

\begin{proof}[Proof of Proposition~\ref{proposition: hte for general groups}]
	We have
	\begin{align*}
		Z^{(U)}_{\beta,\kappa,N}[\gamma] = &\sum_{\sigma \in \Omega_1(B_N,G)} \tr \rho(\sigma(\gamma)) 
		\\&\qquad\cdot\prod_{e \in C_1(B_N)^+} \sum_{\chi_e\in\widehat G} \widehat f_\kappa(\chi_e)\chi_e(\sigma(e)) \prod_{p \in C_2(B_N)^+} \sum_{\chi_p\in\widehat G} \widehat f_\beta(\chi_p)\chi_p(d\sigma(p))
		\\&=
		\sum_{(\chi_e) \in {\widehat G}^{C_1(B_N)^+}} \sum_{(\chi_p) \in {\widehat G}^{C_2(B_N)^+}}
		\prod_{e \in C_1(B_N)} \widehat f_\kappa(\chi_e)  \prod_{p \in C_2(B_N)} \widehat f_\beta(\chi_p) 
		\\&\qquad\cdot \sum_{\sigma \in \Omega_1(B_N,G)} \tr \rho(\sigma(\gamma)) 
		\prod_{e \in C_1(B_N)^+} \chi_e(\sigma(e)) \prod_{p \in C_2(B_N)^+}  \chi_p(d\sigma(p)).
	\end{align*}
	Since characters are multiplicative, it follows that  \( \rho_p(d\sigma(p)) = \prod_{e \in \partial p} \rho_p(\sigma(e)) \) for any \( \sigma \in \Omega_1(B_N,G) \), \( e \in C_1(B_N) \) and \( p \in C_2(B_N), \) and we thus have
	\begin{align*}
		&
		\sum_{\sigma \in \Omega_1(B_N,G)} \tr \rho(\sigma(\gamma)) 
		\prod_{e \in C_1(B_N)^+} \chi_e(\sigma(e)) \prod_{p \in C_2(B_N)^+}  \chi_p(d\sigma(p))
		\\&\qquad=
		\sum_{\sigma \in \Omega_1(B_N,G)} \tr \rho(\sigma(\gamma)) 
		\prod_{e \in C_1(B_N)^+} \chi_e(\sigma(e)) \prod_{p \in C_2(B_N)^+} \prod_{e \in \partial p} \chi_p(\sigma(e)) 
		\\&\qquad= 
		\prod_{e \in C_1(B_N)^+} \sum_{g \in G} (\tr \rho(g))^{\gamma[e]} 
		\chi_e(g)  \prod_{p \in \support \hat \partial e}\chi_{p}(\partial p[e] g)
	\end{align*}
	Noting that for finite abelian groups, we have
	\[
		\sum_{g\in G}\chi(g) = |G| \cdot \mathbf{1}(\chi=1),
	\]
	that the product of two characters is again a character, and that \( g \mapsto \tr \rho (g) \) is a character, we obtain
	\begin{align*}
	Z^{(U)}_{\beta,\kappa,N}[\gamma] &= |G|^{|C_1(B_N)^+|}
		\sum_{\substack{(\chi_e) \in {\widehat G}^{C_1(B_N)^+} \\ (\chi_p) \in {\widehat G}^{C_2(B_N)^+}}}
		\prod_{e \in C_1(B_N)} \widehat f_\kappa(\chi_e)  \prod_{p \in C_2(B_N)} \widehat f_\beta(\chi_p) 
		\\&\qquad\cdot \prod_{e \in C_1(B_N)^+} \mathbf{1} \Bigl( \forall e \in C_1(B_N)^+ \colon   (\tr \rho(g))^{\gamma[e]} 
		\chi_e(g)  \prod_{p \in \support \hat \partial e}\chi_{p}(\partial p[e] g) = 1 \Bigr)
		\\&= 
		|G|^{|C_1(B_N)^+|}
		\sum_{\substack{(\chi_p) \in {\widehat G}^{C_2(B_N)^+}}}
		\prod_{e \in C_1(B_N)} \widehat f_\kappa(\chi_{e,(\chi_p),\gamma} (g) )  \prod_{p \in C_2(B_N)} \widehat f_\beta(\chi_p) .
	\end{align*}
	where 
	\begin{align*}
		\chi_{e,(\chi_p),\gamma} (g) \coloneqq \overline{(\tr \rho(g))^{\gamma[e]} \prod_{p \in \support \hat \partial e}\chi_{p}(\partial p[e] g)   \Bigr)}.
	\end{align*}
	This concludes the proof.
\end{proof}

\end{document}